\documentclass[reqno, 11pt, a4paper, final]{amsart}
\usepackage{amssymb,amsmath,amsthm}

\usepackage[english]{babel}
\usepackage[T1]{fontenc}
\usepackage[utf8]{inputenc}
\usepackage{graphicx}
\usepackage{tikz-cd}
\usepackage{subcaption}
\usepackage{emptypage}
\usepackage{faktor}
\usepackage{enumitem}
\usepackage{mathrsfs}
\usepackage{ragged2e}

	\usepackage[
	top=2.5cm,
	bottom=2.86cm,
	inner=2.6cm, 
	outer=2.6cm,
	footskip=40pt 
	]{geometry}

\usepackage{physics}

\theoremstyle{definition}
\newtheorem{defi}{Definition}[section]

\newtheorem{rmk}[defi]{Remark}

\theoremstyle{plain}
\newtheorem{lm}[defi]{Lemma}
\newtheorem{thm}[defi]{Theorem}

\newcommand{\R}{\mathbb R}
\newcommand{\N}{\mathbb N}
\newcommand{\Z}{\mathbb Z}
\newcommand{\C}{\mathbb C}
\renewcommand{\P}{\mathbb P}

\newcommand{\bx}{\mathbf{x}}
\newcommand{\by}{\mathbf{y}}

\renewcommand{\c}{\mathcal{C}}
\renewcommand{\r}{\mathcal{R}}
\renewcommand{\a}{\mathcal{A}}
\renewcommand{\o}{\mathcal{O}}

\newcommand{\p}{\mathcal{P}}
\newcommand{\f}{\mathcal{F}}
\newcommand{\z}{\mathcal{Z}}
\newcommand{\Or}{\mathcal{O}(\mathcal{R})}

\DeclareMathOperator{\id}{Id}
\DeclareMathOperator{\im}{im}

\DeclareMathOperator{\pr}{pr}
\DeclareMathOperator{\Int}{Int}

\usepackage{lmodern}

\begin{document}
	
	\title{Approximation of holomorphic Legendrian curves with jet-interpolation}
	\author{Andrej Svetina}

	\address[Andrej Svetina]{Faculty of Mathematics and Physics, University of Ljubljana, Jadranska 21, 1000 Ljubljana, Slovenia}
	\email{andrej.svetina@fmf.uni-lj.si}
	
	\date{\today}
	
	

	\begin{abstract}
		We prove several interpolation results for holomorphic Legendrian curves lying in an odd dimensional complex Euclidean space with the standard contact structure. In particular, we show that an arbitrary countable set of points in $\C^{2n+1}$ lies on an injectively immersed isotropic surface with a prescribed complex structure. If the set has no accumulation points, the surface may be taken properly embedded. We also prove a Carleman-type theorem for holomorphic Legendrian curves with interpolation. Namely, a Legendrian curve, defined on a certain type of unbounded closed set in a given open Riemann surface $\r$, may be approximated in the $\c^0$-topology by an entire Legendrian curve with prescribed finite-order Taylor polynomials at a closed discrete set of points in $\r$. Under suitable conditions, the approximating map may be made into a proper embedding.

	\end{abstract}

	\maketitle

\section{Introduction}

	The \emph{standard holomorphic contact structure on $\C^{2n+1}$} is the holomorphic hyperplane subbundle $\xi_{std}$ in the holomorphic tangent bundle $T\C^{2n+1}$ given by $\xi_{std} = \ker \alpha_{std}$, where $\alpha_{std}$ is the holomorphic differential 1-form, called \emph{the standard contact form on $\C^{2n+1}$}, given in the holomorphic coordinates $(x_1,y_1,\ldots,x_n,y_n,z)$ on $\C^{2n+1}$ by
\begin{equation}
	\label{stdform}
\alpha = \dd z + x_1 \dd y_1 + \cdots + x_n \dd y_n.
\end{equation}
From now on we will omit the subscript $std$ and write $\alpha$ for the above 1-form and $\xi = \ker \alpha$. It follows by an easy computation that the top form $\alpha \wedge (\dd \alpha)^n$ is everywhere nonvanishing, i.\ e.\ it is a holomorphic volume form on $\C^{2n+1}$. Hence, the 2-form $\dd \alpha$ is nondegenerate on $\xi$, giving $\xi$ the structure of a holomorphic symplectic bundle, thus the maximal dimension of any complex submanifold in $\C^{2n+1}$, tangent to $\xi$, is $n$. Such subbundles are also called \emph{completely nonintegrable}. In the following we are mainly concerned with tangent complex curves, called \emph{holomorphic Legendrian curves}, that is, holomorphic maps $f\colon \r \to \C^{2n+1}$, defined on open Riemann surfaces, such that $\im \dd f_p \subset \xi_{f(p)}$ holds for every $p \in \r$, or equivalently, $f^*\alpha = 0$ holds along $\r$.

It was proven by Alarcón, Forstnerič and López in \cite{Alarcon2017} that such curves admit Runge- and Mergelyan-type approximation properties in the sense that a holomorphic Legendrian curve defined on a neighbourhood of a holomorphically convex compact set $K$ in an open Riemann surface $M$ may be approximated uniformly on $K$ by proper holomorphic Legendrian embeddings $M \to \C^{2n+1}$. Furthermore, they showed that approximation with complete embeddings is possible on \emph{compact bordered Riemann surfaces} (see Section \ref{cptbrd} for the definition), in the sense that any Legendrian curve $f\colon M \to \C^{2n+1}$, defined on a compact bordered Riemann surface $M$ with nonempty boundary $bM$ and of class $\a^1(M)$, may be approximated uniformly on $M$ by continuous injective maps $F\colon M \to \C^{2n+1}$ such that $F$ restricted to the interior $\Int M = M \backslash bM$ is a complete holomorphic Legendrian embedding, whereby an embedding $F\colon \Int M \to \C^{2n+1}$ is complete if the standard Euclidean metric on $\C^{2n+1}$ induces a complete metric on the image $F(\Int M)$.

It follows from the former result that any open Riemann surface properly embeds into $\C^{2n+1}$ as a holomorphic Legendrian curve. We improve this result to show that the proper embedding in question hits a prescribed closed discrete set in the codomain.

\begin{thm}
	\label{weierstrass}
	Let $\{a_n\}_{n\in \N}$ be a closed discrete set in an open Riemann surface $\r$. For an arbitrary closed discrete set $\{b_n\}_{n\in \N}$ in $\C^{2n+1}$ there exists a proper holomorphic Legendrian embedding $f\colon \r \to \C^{2n+1}$ such that $f(a_n) = b_n$ holds for every $n \in \N$.
	
\end{thm}

Theorem \ref{weierstrass} follows from the more precise statement given in Theorem \ref{mergelyan} below. Namely, in addition to prescribing the values to the Legendrian curve one may also prescribe finite order Taylor polynomials, i.\ e.\ \emph{jets}, provided these are obtained as jets of holomorphic Legendrian curves.

Note that Legendrian \emph{immersions} are a special case of \emph{directed immersions} $f\colon M \to X$. Let $G_k(X)$ denote the Grassmann bundle of $k$-planes in the complex manifold $X$. For a holomorphic immersion $f\colon M \to X$, where $M$ is a complex manifold of dimension $\dim M =k$, define its \emph{Gauss map} $Gf\colon M \to G_k(X)$ by
\[
Gf\colon p \mapsto \left(f(p), \dd f_p(T_pM)\right) \quad \text{for} \; p\in M.
\]
Let $A \subset G_k(X)$ be an arbitrary subset and denote by $A_q = A \cap \pi^{-1}(q)$ its fiber over $q \in X$, where $\pi \colon G_k(X) \to X$ is the canonical projection. An immersion $f\colon M \to X$ is said to be \emph{$A$-directed}, if its image lies in $A$, i.\ e.\ $\dd f_p(T_pM) \in A_{f(p)}$ for all $p \in M$. In the special case where $X = \C^n$ and $A$ takes the form $A=\C^n\times \mathcal{A}$ where $\mathcal{A}\subset \C^n$ is a closed conical subvariety such that $\mathcal{A}\backslash \{0\}$ is an Oka manifold (see \cite{Forstneric2017} for a comprehensive account on Oka manifolds), the Mergelyan approximation property for immersions of open Riemann surfaces into $\C^m$ was proven by Alarcón and Forstnerič in \cite{Alarcon2014}. Furthermore, the Carleman approximation property of directed immersions in the same special case has been considered by Castro-Infantes and Chenoweth in \cite{CastroInfantes2020}. There, the authors have shown that an $A$-directed immersion $f\colon S \to \C^n$, where $A$ is as before and $S$ is a \emph{Carleman admissible subset} (see Definition \ref{carlad}) in an open Riemann surface $\r$, may be approximated by $A$-directed immersions $F\colon\r \to \C^n$, such that $|F(q)-f(q)|< \varepsilon(q)$ holds for $q \in S$, where $\varepsilon\colon S \to (0,\infty)$ is a given positive continuous function.

If one wishes to establish a similiar approximation property in the case of Legendrian curves, the situation is markedly different, since the fibers $A_p$ are allowed to vary from point to point. However, once one establishes the Runge (or Mergelyan) approximation property, Carleman property follows by applying a simple induction procedure similar to that in \cite{CastroInfantes2020}. Interpolation on a closed subset $\Lambda$ is obtained by interpolating on successively bigger finite subsets, namely the intersection of $\Lambda$ with holomorphically convex sets in a chosen normal exhaustion of $\r$. This presents a rough outline of the proof of Theorem \ref{carleman} whose special consequence is the following.

\begin{thm}
	Suppose $f\colon \R \to \C^{2n+1}$ is a smooth proper arc in $\C^{2n+1}$ satisfying $f^*\alpha = 0$ and $\eta\colon \R \to \r$ is a smooth proper arc in an open Riemann surface $\r$ with no self-intersections. For an arbitrary continuous function $\varepsilon\colon \R \to (0,\infty)$ there exists a proper holomorphic Legendrian embedding  $F\colon \r \to \C^{2n+1}$ such that $\|F(\eta(t)) - f(t)\| < \varepsilon(t)$ for every $t \in \R$.
\end{thm}

\section{Preliminaries and main results.}

A smooth Jordan arc in an open Riemann surface $\r$ is the image of an injective proper smooth map $\gamma\colon  I \to \r$, defined on an interval $I \subset \R$ (open, closed or half-open), with no self intersections. A smooth Jordan closed curve in $\r$ is the image of a smooth embedding $S^1 \to \r$ of the circle. Note that $I$ does not contain an endpoint if and only if the corresponding curve $\gamma$ is divergent near that endpoint.

A compact subset $S$ of an open Riemann surface $\mathcal{R}$ is called \emph{admissible} if $S = K \cup \Gamma$, where $K=\bigcup \overline{D}_j$ is a union of finitely many pairwise disjoint compact domains $\overline{D}_j$ in $\mathcal{R}$ with piecewise $\c^r$-smooth boundary for some $r \in \N$ and $\Gamma = \bigcup \Gamma_i$ is a union of finitely many pairwise disjoint smooth Jordan arcs or closed curves that intersect $bK$ only at their endpoints (if at all) and  their intersections with the boundary of $K$ are transverse. Note that by the above definition every arc in $\Gamma$ is defined on a \emph{compact} interval in $\R$.

If $S$ is an admissible set in an open Riemann surface $\r$, we use the following notation for the set of $r$-times continously differentiable functions on $S$ that are holomorphic in the interior $\Int S$ of $S$:
$$
\a^r(S) := \{f \in \c^r(S,\C)\,|\, f\colon  \Int S \to \C \text{ is holomorphic}\}.
$$
A map $f\colon S \to \C^{2n+1}$ defined on some admissible set $S$ in an open Riemann surface $\r$ is called a \emph{generalised Legendrian curve} of class $\a^r(S)$ if its component functions are of class $\a^r(S)$ and $f^*\alpha=0$ holds at every point of $S$, where $\alpha$ is as in \ref{stdform}. Our first main result is the following:

\begin{thm}
	\label{mergelyan}
	Let $S$ be a Runge admissible subset in an open Riemann surface $\r$ and let $f\colon S \to \C^{2n+1}$ be a generalised Legendrian curve of class $\a^r(S)$ for some $r \geq 1$. Suppose we are given
	\begin{enumerate}[label=\roman*)]
		\item a closed discrete subset $\Lambda = \Lambda' \cup \Lambda'' \subset \mathcal{R}$, where $\Lambda' \subset \Int S$ and $\Lambda'' \subset \mathcal{R}\backslash S$, 
		
		\item a holomorphic Legendrian curve $\phi\colon O \to \C^{2n+1}$, defined on a neighbourhood $O \subset \r \backslash S$ of the set $\Lambda''$, and
		
		\item a function $m\colon  \Lambda \to \N \cup \{0\}$.
	\end{enumerate}
	Then $f$ may be approximated uniformly on $S$ in the $\c^r(S)$-topology by holomorphic Legendrian curves $F \colon \mathcal{R} \to \C^{2n+1}$ such that the following hold:
	\begin{enumerate}[label=(\Roman*)]
		\item the map $F-f$ has a zero of order at least $m(p)$ at $p$ for all $p \in \Lambda'$, and
		\item the map $F-\phi$ has a zero of order at least $m(p)$ at every point $p \in \Lambda''$.
	\end{enumerate}

	Let $\tilde{f} \colon  S \cup O \to \C^{2n+1}$ be given by $\tilde{f}(q) = f(q)$ for $q \in S$ and $\tilde{f}(q) = \phi(q)$ for $q \in O$. If $\dd \tilde{f}(p) \neq 0$ for all $p \in \Lambda \cap \{m\geq 1\}$, then $F$ may be chosen an immersion. If, in addition, the map $\tilde{f}|_\Lambda$, is injective, then $F$ may be chosen injective. If $\tilde{f}|_\Lambda$ is proper, then $F$ may be made proper.
\end{thm}

Our result is similar to the one obtained by Alarcón, Forstnerič and Lárusson in \cite{Alarcon2019b}. There the authors have proven that any holomorphic Legendrian curve $f\colon  V \to \C\P^3$, where $V \subset M$ is a neighbourhood of a compact set $K \subset M$ in an open or compact Riemann surface $M$, may be approximated uniformly on $K$ by entire Legendrian curves $F\colon  M \to \C\P^3$, interpolating $f$ at a finite number of points in $K$ to a given finite order. It follows from their proof that in the case the image $f(V)$ is contained in some $\C^3 \subset \C\P^3$, the image $F(M)$ also lies in $\C^3$, thus giving a special case of Theorem \ref{mergelyan} when $n=1$, $\Lambda'' = \emptyset$ and $f$ is holomorphic on a neighbourhood of $S$. Their method uses a representation formula for Legendrian curves in terms of meromorphic functions which can in turn be approximated by entire meromorphic functions by Royden's theorem (see \cite{Royden1967}) from which one recovers the required entire Legendrian curve. Interpolation on a discrete set can then be obtained by using an inductive procedure.

Note that in our setting, one may first use Mergelyan's theorem with interpolation in order to approximately extend the components of the curve from the given compact set to some neighbourhood of it while fixing its jets on the required finite set and then use the result from \cite{Alarcon2019b} to further approximate the given Legendrian curve with an entire one. However, in the first step, one would still need to construct an appropriate holomorphic spray in order to guarantee the curve on the neighbourhood remains Legendrian. Since this construction constitutes the bulk of our proof, we present it in its entirety.

Note also that Theorem \ref{mergelyan} implies that under necessary conditions on the jets, prescribed at the set $\Lambda$, one may approximate the given Legendrian curve with a \emph{proper} Legendrian embedding, where the necessary conditions are the obvious ones, namely that the prescribed values are pairwise distinct and the prescribed first order derivatives are nonvanishing. We also do not impose any restrictions on the set at which we interpolate, apart from it being closed and discrete as dictated by the identity priciple.

Our second result concerns approximation on a certain kind of unbounded sets which are obtained as a generalisation of the following well known definition: a compact set $K$ in an open Riemann surface $\mathcal{R}$ is $\mathcal{O}(\mathcal{R})$-convex (or \emph{Runge in $\mathcal{R}$}), if its \emph{holomorphically convex hull}
$$
\widehat{K} := \left\{z \in \mathcal{R} \, : \, |f(z)| \leq \sup_{u \in K} |f(u)| \quad \text{for all} \;\, f \in \mathcal{O}(\mathcal{R})\right\}
$$
is compact. By the classical result of Behnke and Stein, see \cite{Behnke1947}, this holds if and only if $K$ has no \emph{holes}, that is, relatively compact connected components in the complement $\mathcal{R}\backslash K$. Following \cite{Manne2011} we extend the definition to certain unbounded sets. Let $E$ be a closed subset of $\mathcal{R}$ and define its \emph{hull} to be the set 
$$
\widehat{E} = \bigcup_{i=1}^{\infty} \widehat{E}_j,
$$
where $(E_j)_{j \in \N}$ is some normal exhaustion, meaning $E_j \subset \Int E_{j+1}$ for all $j$ of $E$ by compact sets $E_j$. Note that for a given closed set $E$ its hull $\widehat{E}$ is independent of the choice of the normal exhaustion: if $\{E_j'\}_{j \in \N}$ is another such exhaustion, then $E_j' \subset E_{k_j}$ for some $k_j \in \N$ and thus
\[
\bigcup_{j=1}^\infty \hat{E}_j' \subset \bigcup_{j=1}^\infty \hat{E}_{k_j} \subset
\bigcup_{k=1}^\infty \hat{E}_k
\]
and the reverse inclusion holds by symmetry. If $\widehat{E} = E$, the set $E$ is called $\o(\r)$-convex. For a closed set $E$ define the holes $h(E)$ of $E$ by
$$
h(E) = \overline{\widehat{E}\backslash E}.
$$
A closed set $E$ in an open Riemann surface $\mathcal{R}$ has \emph{bounded exhaustion hulls} if for every compact set $K \subset \mathcal{R}$ the set $h(K\cup E)$ is compact. Since open Riemann surfaces are Stein manifolds, every compact $\o(\r)$-convex subset has bounded exhaustion hulls. We will use the following lemma \cite[Lemma 3]{Chenoweth2019}.

\begin{lm}
	\label{bddexhaustion}
	Let $X$ be a Stein manifold and let $E \subset X$ be a closed $\o(X)$-convex set with bounded exhaustion hulls. Then there exists a normal exhaustion $\{K_j\}_{j=1}^\infty$ of $X$ by compact $\o(X)$-convex sets $K_j$ such that $K_j \cup E$ is $\o(X)$-convex for every $j \in \N$.
\end{lm}

In the case of totally real submanifolds in complex manifolds, the two properties give precisely the type of sets on which Carleman approximation is possible (see for example \cite{Manne2011} and \cite{Chenoweth2019} for the case of totally real sets in Stein manifolds and approximation where the target manifold is Oka, respectively). The next definition is due to Castro-Infantes and Chenoweth (see \cite{CastroInfantes2020}).

\begin{defi}
	\label{carlad}
	A \emph{Carleman admissible set} $S$ in an open Riemann surface $\r$ is a closed $\o(\r)$-convex set with bounded exhaustion hulls of the form $S = K \cup E$ where $K$ is a union of a locally finite collection of pairwise disjoint compact domains with piecewise $\c^r$-smooth boundaries for some $r \in \N$ and $E = \overline{S\backslash K}$ is the union of a locally finite collection of pairwise disjoint smooth embedded arcs that intersect the boundary $bK$ only at their endpoints, if at all, and those intersections are transverse.
\end{defi}

A key part of the proof consists of using Mergelyan's theorem with jet interpolation in order to construct appropriate holomorphic functions on open Riemann surfaces. The background on this result as well as the proof is given in the survey \cite[Section 4]{Fornaess2018} on holomorphic approximation theory.

By using an induction procedure we are also able to obtain the next result on approximation in the fine Whitney topology and jet-interpolation of Legendrian curves defined on certain unbounded sets in open Riemann surfaces. Again we are able to approximate the given Legendrian curve with a proper holomorphic Legendrian embedding, provided the prescribed jets satisfy the obvious necessary condition. Note that in this case the approximation takes place in the $\c^0$-topology as opposed to the $\c^r$-topology as in Theorem \ref{mergelyan}. However, this is likely due to the method used in the proof and not some underlying obstruction.

\begin{thm}
	\label{carleman}
	Let $S = K \cup E$ be a Carleman admissible subset in an open Riemann surface $\r$ and let $f\colon  S \to \C^{2n+1}$ be a generalised Legendrian curve of class $\a^1(S)$. Suppose we are given:
	\begin{enumerate}[label=(\roman*)]
		\item a strictly positive continuous function $\varepsilon\colon  S \to (0,\infty)$;
		\item a closed discrete subset $\Lambda = \Lambda' \cup \Lambda''$ of $\r$ with $\Lambda' \subset \Int S = \Int K$ and $\Lambda'' \subset \r \backslash S$;
		\item a holomorphic Legendrian curve $\phi\colon O \to \C^{2n+1}$, defined on an open neighbourhood $O \subset \r \backslash S$ of the set $\Lambda''$ and
		\item a function $m\colon  \Lambda \to \N \cup \{0\}$.
			\end{enumerate}

		Then there exists a holomorphic Legendrian curve $F\colon  \r \to \C$ satisfying the following conditions:
		\begin{enumerate}[label=(\Roman*)]
			\item $\|F(q) - f(q)\| < \varepsilon(q)$ for all $q \in S$, i.\ e.\ $F$ is $\varepsilon$-close to $f$ in the fine $\c^0(S)$-topology, where $\|\cdot \|$ denotes the Euclidean norm on $\C^{2n+1}$,
			\item $F-f$ has a zero of order at least $m(q)$ at all points $q \in \Lambda'$, and
			\item $F-\phi$ has a zero of order at least $m(p)$ at $p$ for all $p \in \Lambda''$.
		\end{enumerate}
		
		Let $\tilde{f} \colon  S \cup O \to \C^{2n+1}$ be given by $\tilde{f}(q) = f(q)$ for $q \in S$ and $\tilde{f}(q) = \phi(q)$ for $q \in O$. If $\dd \tilde{f}(p) \neq 0$ for all $p \in \Lambda\cap \{m\geq 1\}$, then $F$ may be chosen an immersion. If, in addition, the map $\tilde{f}|_\Lambda$ is injective, then $F$ may be chosen injective. If $\tilde{f}|_{S \cup \Lambda''}$ is proper, then $F$ may be made proper.
\end{thm}


\section{Mergelyan approximation with jet-interpolation}

In this section we prepare the necessary ingredients for the proofs of Theorems \ref{mergelyan} and \ref{carleman} which are explained in the next sections. We first show how to add interpolation to the approximation result for generalised Legendrian curves on admissible sets from \cite{Alarcon2017}. In particular, we show that a generalised Legendrian curve, defined on some admissible set in an open Riemann surface, may be approximated in the $\c^r$-topology with holomorphic Legendrian immersions while fixing it to a given finite order at a finite number of points inside its domain and prescribing a finite order holomorphic Legendrian Taylor polynomial for it at a finite order of points outside of its domain. Next, we show this kind result may be improved to approximation with embeddings. Finally, we show how to control the norm of the approximating curve at the boundary of some domain slightly larger than the starting admissible set, which will be a crucial step in proving that generalised Legendrian curves may be approximated with entire embedded Legendrian curves which interpolate it on a given closed dicrete set.

\subsection{Jet interpolation on admissible sets}
We begin by first showing that interpolation to a given finite order is possible on finite subsets in the interior of the domain of the approximated curve. The following is a version of \cite[Lemma 4.3]{Alarcon2017} with added interpolation.

\begin{lm}
	\label{finiteInterpolation}
	Let $S = K \cup \Gamma$ be an admissible subset in an open Riemann surface $\mathcal{R}$. Suppose $R$ is a compact domain with smooth boundary in $\mathcal{R}$ such that $S \subset \Int R$ and $R$ deformation retracts onto $S$. Let $\Lambda \subset \Int S = \Int K$ be a finite subset not intersecting $bK \cup \Gamma$, and let $m\colon \Lambda \to \N \cup \{0\}$ be a function. For any generalised Legendrian curve $f\colon S \to \C^{2n+1}$ of class $\mathcal{A}^r(S)$, $r\geq 1$, there exists a holomorphic Legendrian curve $F\colon  R \to \C^{2n+1}$, defined on some neighbourhood of $R$ in $\r$, which is $\mathcal{C}^r$-close to $f$ on $S$, such that the map $F-f$ has a zero of order at least $m(p)$ for all $p \in \Lambda$. If $\dd f(p) \neq 0$ for all $p \in \Lambda \cap \{m \geq 1\}$, then $F$ may be chosen to be an immersion. If for some $1\leq j \leq 2n+1$ the $j$-th coordinate function $\pr_j \circ f$ of $f$ in $\C^{2n+1}$ is already holomorphic on a neighbourhood of $R$, the approximating curve $F$ may be taken such that $\pr_j \circ F =\pr_j\circ f$.
\end{lm}

\begin{rmk}
	Let $h\colon  \mathcal{R} \to (\C^3,\alpha)$ be a holomorphic Legendrian curve with components \linebreak $f=(x,y,z)\colon u \mapsto (x(u),y(u),z(u))$. From the Legendrian condition
	$$
	\dd z +x\dd y= 0, \quad \text{i.\ e.\ } \quad \dot{z}=-x\dot{y}
	$$
	it follows that for any Legendrian curve $g=(x',y',z')\colon  \mathcal{R} \to \C^3$ the curve $f$ has the same $m$-jet at $p \in \mathcal{R}$ as $g$, if $z(p) = z'(p)$ and the maps $(x,y),(x',y')\colon  \mathcal{R} \to \C^2$ have the same $m$-jet at $p$.	Moreover, such a Legendrian curve is immersed if and only if its $(x,y)$-projection is.
\end{rmk}

\begin{rmk}
	\label{cntMf}
	It follows from the Legendrian condition that a Legendrian curve in $\C^{2n+1}$ is nonconstant if and only if one of its $x_i$ or $y_i$ components is nonconstant. Denote the coordinates on $\C^{2n+1}$ by $x_i, y_i, z$, $i=1,\ldots,n$. Note that the maps
	\begin{align*}
		c_1\colon 
		(x_1,\ldots,x_n,y_1,\ldots,y_j,\ldots,y_n,z) & \mapsto (x_1,\ldots,x_n,y_j,\ldots,y_1,\ldots,y_n,z)\\
		c_2\colon 
		(x_1,\ldots,x_n,y_1,\ldots,y_n,z) &\mapsto (x_1,\ldots,x_n,-y_1,\ldots,y_n,z+x_1y_1)
	\end{align*}
	are $\C$-linear isomorphisms of $\C^{2n+1}$. It is immediate that the 1-forms
	\begin{align*}
		\alpha_1 &=
		\dd z + x_1 \dd y_j + x_j \dd y_1 +
		\sum_{i \notin \{1,j\}} x_i \dd y_i \\
		\alpha_2 &=
		\dd z + y_1 \dd x_1 + \sum_{i\neq 1} x_i \dd y_i
	\end{align*}
	are contact forms and $c_i^* \alpha_i = \alpha$ holds. When needed, we may thus interchange the roles of $x_i$ and $x_j$, $y_i$ and $y_j$, or $x_i$ and $y_j$. 
\end{rmk}

\begin{proof}[Proof of Lemma \ref{finiteInterpolation}]
	We assume $R$ (and hence its deformation retract $S$) is connected, otherwise we construct the approximating curve on each connected component of $R$ separately. Moreover, since the curve we wish to construct only needs to be defined on some neighbourhood of $R$, we may assume $\r$ deformation retracts onto $R$, thus $R$ is Runge in $\r$.
	
	Write $f \colon S \to \C^{2n+1}$ in coordinates as $f=(x_1',\ldots,x_n',y_1',\ldots,y_n',z') = (\bx', \by', z')$. We follow the proof of \cite[Lemma 4.3]{Alarcon2017} with a minor modification to the construction of the holomorphic spray, ensuring the interpolation condition for all suitable parameters. We will thus construct a holomorphic spray $\tilde{x}_1\colon  \r \times \C^s \times \C^N \to \C^s \times \C^N \ni (\zeta,\xi)$, $s$ and $N$ to be defined later on, of the form
	
	$$
	\tilde{x}_1(u,\zeta,\xi) = 
	x_1(u) + 
	\sum_{j=1}^{s}\zeta_j g_j(u)+
	\sum_{p\in \Lambda}\xi_{p}h_{p}(u)
	+ \delta w(u),
	\quad
	u \in \r,
	$$
	where $x_1$ is a Mergelyan approximation of the $x_1'$-component of $f$, $g_j$'s will be taking care of correcting the periods over a suitably chosen family of curves generating the homology basis of $S$ (and thus of $R$ and $\r$) so that the 1-form $\bx \dd \by$ becomes exact, $h_p$'s will correct the integrals of the 1-form $\bx \dd\by$ in order to interpolate the $z'$-component, and finally $w$ will get rid of (some) zeroes of the 1-form $\dd x_1$ in order to make the approximating map an immersion. Here, $\delta$ will be a suitably chosen parameter $\delta \in \C^* = \C \backslash \{0\}$.

	
	
	
	\textbf{Step 1:} Since $S$ is a deformation retract of $\r$ it is Runge in $\r$, thus we may use Mergelyan's theorem with jet-interpolation (see e.\ g.\ \cite[Theorem 16]{Fornaess2018} and \cite[Theorem 1.12.11]{Alarcon2021}) to approximate the coordinate functions $ \by'=(y_1',\ldots,y_n') \colon  S \to \C^n$ with holomorphic functions $\by = (y_1,\ldots,y_n) \colon  \mathcal{R} \to \C^n$ in the $\c^r(S)$-topology such that for all $j = 1,\ldots,n$ and all $p \in \Lambda$:
	$$
	(y_j')^{(k)}(p) - y_j^{(k)}(p) = 0, \quad \text{for all} \; k=0,\ldots, m(p).
	$$
	Do the same for $\bx' = (x_1',\ldots,x_n')$ and obtain $\bx = (x_1,\ldots,x_n) \colon  \r \to \C^n$. If $y_1$ is constant, we may add a point $q \in \r \backslash S$ to $S$, obtaining the still admissible set $S \cup \{q\}$ in $\r$ with no holes, and prescribe a different value for the $y_1'$ component at $q$, thus obtaining a nonconstant approximating function $y_1$.
	


	
	Suppose now that every point $p \in \Lambda \cap \{m \geq 1\}$ is a regular point of $f$, meaning that $\dd f_p$ is nonvanishing for all $p \in \Lambda \cap \{m\geq 1\}$. This is a necessary condition if we wish to approximate $f$ with an \emph{immersion}, interpolating $f$ to a given finite order at the points $p \in \Lambda$. By the previous argument the function $y_1\colon \r \to \C$ is nonconstant and thus the set of zeroes of $\dd y_1$ is discrete. Moreover, since $R$ is relatively compact, the set $Z:=R \cap \{\dd y_1 = 0\}$ is finite and by slightly enlarging $R$ if necessary we may assume $bR \cap \{\dd y_1 =0\}$ is empty. If the last condition in the theorem holds ($f$ is an immersion on some neighbourhood of $\Lambda\cap \{m\geq 1\}$), we need only correct the approximating map $(\bx,\by)$ outside of $\Lambda \cap \{m\geq 1\}$, that is on $Z' := Z \backslash (\Lambda\cap \{m \geq 1\})$. Choose a holomorphic function $\eta \colon  \r \to \C$ such that
	\begin{enumerate}[label=\roman*)]
		\item \label{ena} $\eta^{(k)}(p) = 0$ for all $p \in \Lambda \cap \{m\geq 1\}$, $k = 0,\ldots,m(p)$,
		\item \label{dva} $\dd\eta(p) \neq 0$ if $p \in Z'$ and $\dd x_1(p) = 0$,
		\item \label{tri} $\dd\eta(p) = 0$ if $p \in Z'$ and $\dd x_1(p) \neq 0$.
	\end{enumerate}
	Such a function exists by Florack's result on Weierstrass-type approximation of holomorphic functions with jet-interpolation on open Riemann surfaces, see \cite{Florack1948} and \cite[Theorem 1.12.14]{Alarcon2021} for more details. Note that a suitably chosen neighbourhood of the set $Z'$ consists of finitely many disjoint disks and is thus admissible in $\mathcal{R}$.
	
	Define $x_1^\delta \colon  \mathcal{R} \to \C$ by
	$$
	x_1^\delta(u) = x_1(u) + \delta \eta(u), \quad \delta \in \C^*.
	$$
	For all $\delta \neq 0$ we have $\dd x_1^\delta(p) \neq 0$ whenever $\dd y_1(p) = 0$ for all $p \in R$ by properties \ref{dva} and \ref{tri}, thus the map 
	\[
	(x_1^\delta, y_1,x_2,y_2,\ldots,x_n,y_n) \colon  \mathcal{R} \to \C^{2n}
	\]
	is an immersion on $R$ for all $\delta \in \C^*$. If the parameter $\delta$ is chosen small enough, the map $x_1^\delta$ with $\delta$ now fixed approximates well the map $x_1$ in the $\c^r(S)$-topology. If the immersivity condition is not needed, we put $\delta = 0$.



	
	\textbf{Step 2:} 
	By \cite[Lemma 1.12.10]{Alarcon2021} there exists a connected Runge admissible set $C$ in $S$ consisting of finitely many closed curves $C_1,\ldots, C_s \subset S$ generating the first homology group of $S$ (that is, $\iota_*\colon H_1(C) \to H_1(S)$ is an isomorphism, where $\iota\colon  C \to S$ is the inclusion). Furthermore, the lemma ensures each curve $C_i$ contains a nontrivial Jordan arc $\widetilde{C}_i$ not intersecting $\cup_{j\neq i}C_j$. Since $bS \cap \Lambda = \emptyset$ by assumption, we may slightly deform the curves $C_i$, so that $C$ does not intersect $Z \cup\Lambda$.
	
	Choose a smoothly bounded relatively compact neighbourhood $U \subset \r$ of the set $C$ containing no point in $Z \cup \Lambda$, such that $U$ deformation retracts onto $C$.	Fix a $p_0 \in U$. For every $p \in Z \cup \Lambda$ choose a smoothly bounded closed embedded disk $\Omega_p \subset \mathcal{R}\backslash \overline{U}$, centered at $p$ such that $\Omega_p \cap \Omega_q = \emptyset$ for all $p,q \in Z \cup \Lambda$, $p \neq q$, and $\Omega_p \subset (\Int K) \cup (\r \backslash S)$. Set $\Omega = \cup_{p \in Z \cup \Lambda} \Omega_p$. For each $p \in Z\cup\Lambda$ choose a smooth embedded arc $E_p$, connecting the point $p_0$ to the point $p$ such that
	\begin{enumerate}[label = \roman*)]
		\item $E_p$ and $b\overline{U}$ intersect transversely and in a single point, which we label $u_p$; 
		\item $E_p$ and $b\Omega_p$ intersect transversely and in a single point, which we label $\omega_p$;
		\item $(\Omega_p \cup E_p) \cap (\Omega_q \cup E_q) = \emptyset$ for any distinct points $p,q \in Z\cup \Lambda$.
	\end{enumerate}
	
	Let $C' = \overline{U} \cup (\cup_{p \in Z\cup\Lambda}E_p \cup \Omega_p)$. By replacing $E_p$ with the subarcs $E_p'$, connecting $u_p$ to $\omega_p$, we obtain a compact admissible set, still denoted $C'$, which deformation retracts onto $\overline{U}$ and hence is Runge in $\r$. Note that $\Lambda \subset \Int C'$.
	
	
	
	
	\textbf{Step 3:}
	Construct smooth functions $g_j\colon C\to \C$ with support in $\widetilde{C}_j$ for $j=1,\ldots,s$ satisfying
	\begin{enumerate}[label=G$_{\arabic*}$:, ref=G$_{\arabic*}$]
		\item \label{intg=delta} $\int_{C_i}g_j\dd y_1 = \delta_{ij}$ for all $i,j = 1,\ldots,s$.
	\end{enumerate}
	Use Mergelyan's theorem (recall that $C$ is Runge in $S$ and thus also in $\r$) to approximate $g_j$ with a holomorphic function on $\r$, still denoted by $g_j$, such that $\int_{C_i}g_jdy \approx \delta_{ij}$. For each $j$ extend $g_j|_{\overline{U}}$ to a function $g_j'$ of class $\a^r(C')$ such that
	\begin{enumerate}[label=G$_{\arabic*}$:, ref=G$_{\arabic*}$]
		\setcounter{enumi}{1}
		\item \label{gom=0} $g_j'|_{\Omega_p} \equiv 0$ for all $p \in Z \cup \Lambda$ and
		\item \label{gz=0} $\int_{E_p}g_j '\dd y_1 = 0$.
	\end{enumerate}

	Since $g_j'$ is constant on $\Omega$, it is holomorphic on $\Int C'$ and is thus of class $\a^r(C')$. Since $C'$ is an admissible Runge subset of $\r$, by using Mergelyan's theorem we may approximate $g_j'$ in $\c^r(C')$-topology with a holomorphic function $\tilde{g}_j\colon \r \to \C$, interpolating to order $m(p)$ at all points $p \in \Lambda \cap \{m\geq 1\} \subset \Int C'$ and to order at least 1 at the points $p \in Z'\cup (\Lambda \cap\{m=0\}) \subset \Int C'$. In other words, the following conditions hold for $j = 1, \ldots, s$:
	\begin{enumerate}[label=G$_{\arabic*}$:, ref=G$_{\arabic*}$]
		\setcounter{enumi}{3}
		\item \label{dgjp=0} $\tilde{g}_j(p)=0$ and $\dd \tilde{g}_j(p) = 0$ for $p \in Z'\cup (\Lambda \cap\{m=0\})$,
		\item \label{glam=0} $\tilde{g}_j^{(k)}(p) = 0$ for all $p \in \Lambda\cap \{m\geq 1\}$, $k = 0,\ldots m(p)+1$.
	\end{enumerate}
	
	

	\textbf{Step 4:}
	Let $\Omega = \cup_{p \in Z \cup \Lambda} \Omega_p$. Construct functions $h_p \colon  C' \to \C$ of class $\mathcal{A}^r(C')$ for $p \in \Lambda$, satisfying the conditions:
	\begin{enumerate}[label=H$_{\arabic*}$:, ref=H$_{\arabic*}$]
		\item \label{hom=0} $h_p \equiv 0$ on $\overline{U} \cup \overline{\Omega}$,
		\item \label{hz=1} $\int_{E_p}h_q \dd y_1 = \delta_{pq}$ for all $p,q \in \Lambda$.
	\end{enumerate}
	Note that $h_p$ is constant on $\Int C' = U \cup \Omega$ and is thus holomorphic there. Again, by using Mergelyan's theorem and the fact that $C'$ is Runge in $\mathcal{R}$, we may assume $h_p$ are holomorphic on $\mathcal{R}$, while the above equalities hold approximately and as closely as desired. Using jet-interpolation we ensure
	\begin{enumerate}[label=H$_{\arabic*}$:, ref=H$_{\arabic*}$]
		\setcounter{enumi}{2}
		\item \label{dhp=0} $h_p(q) = 0$ and $\dd h_p(q) = 0$ for all $p \in \Lambda$,  $q \in Z' \cup (\Lambda \cap \{m=0\})$.
		\item \label{hlam=0} $h_p^{(k)}(q) = 0$ holds for all $p \in \Lambda$, $q \in \Lambda \cap \{m\geq 1\}$ and $k = 0,1,\ldots,m(p)+1$,
	\end{enumerate}

	
	
	\textbf{Step 5:}
	Define the spray
	$$
	\tilde{x}_1(u,\zeta,\xi) = 
	x_1(u) + 
	\sum_{j=1}^{s}\zeta_j \tilde{g}_j(u)+
	\sum_{p\in \Lambda}\xi_{p}h_{p}(u), 
	\quad u \in \r, 
	\quad \zeta \in \C^s,
	\quad \xi \in \C^{|\Lambda|},
	$$
	where $\zeta = (\zeta_1,\ldots,\zeta_s)$ and $\xi = (\xi_p)_{p\in \Lambda}$. Denote by $\tilde{\bx}(\zeta,\xi)$ the map
	$$
	\tilde{\bx}(\zeta,\xi) = (\tilde{x}_1(\cdot,\zeta,\xi),x_2,\ldots,x_n) \colon  \mathcal{R} \to \C^n.
	$$
	By replacing $x_1$ with $x_1^\delta$ as constructed in Step 1 we obtain the spray $\tilde{x}_1^\delta\colon \mathcal{R}\times \C^s \times \C^{|\Lambda|} \to \C$:
	$$
	\tilde{x}_1^\delta(u,\zeta,\xi) = 
	x_1^\delta(u) + 
	\sum_{j=1}^s \zeta_j\tilde{g}_j(u)+
	\sum_{p\in \Lambda}\xi_{p}h_{p}(u), 
	\quad u \in \r, 
	\quad \zeta \in \C^s,
	\quad \xi \in \C^{|\Lambda|}.
	$$
	Since $\dd y_1$ is nonvanishing on $bR$ by construction, there exists a neighbourhood $\r'$ of the compact set $R$ such that all the zeroes of $\dd y_1$ in $\r'$ lie in $\Int R$, i.\ e.\ we have 
	\[\{u \in \r': \; \dd y_1(u) = 0\} = Z.\]
	It follows that for all $\zeta, \xi \in \C$ the map
	$$
	\r'\backslash (\Lambda \cap \{m\geq 1\}) \ni u \mapsto (\tilde{x}_1^\delta(u,\zeta,\xi),y_1(u)) \in \C^2
	$$
	is an immersion: for all $u \in \r' \backslash (\Lambda \cap \{m\geq 1\})$ either $\dd y_1(u) \neq 0$ or $u \in Z'$, but then
	$$
	\dd \tilde{x}_1^\delta(u,\zeta,\xi) = \dd x_1^\delta(u) +
	\sum_{j=1}^s \zeta_j\dd\tilde{g}_j(u)+
	\sum_{p\in \Lambda}\xi_{p}\dd h_{p}(u)
	= \dd x_1^\delta(u)
	\neq 0,
	$$
	since $\dd\tilde{g}_j(u) = 0$ for $j=1,\ldots,s$ by \ref{dgjp=0} and $\dd h_p(u) = 0$ for all $p \in \Lambda$ by \ref{dhp=0}. Moreover, conditions \ref{glam=0} and \ref{hlam=0} ensure $\tilde{x}_1^\delta(\cdot, \zeta,\xi)\colon \r \to \C$ agrees with $x_1^\delta\colon  S \to \C$ to order at least $m(p)$ at all points $p \in \Lambda$ for all $(\zeta,\xi) \in \C^s \times \C^{|\Lambda|}$, thus the map $\r' \ni u \mapsto (\tilde{\bx}^\delta(u,\zeta,\xi),\by(u)) \in \C^{2n}$ with $\delta \neq 0$ is an immersion for all parameters $(\zeta,\xi) \in \C^s\times \C^{|\Lambda|}$ if and only if the last condition in the theorem holds (i.\ e.\ $f$ is an immersion near $p \in \Lambda \cap \{m\geq 1\}$).
	
	Define the extended period map $\widetilde{\mathcal{P}} = (\mathcal{P}, \mathcal{Z}) \colon  \mathcal{A}^r(S)^{2n} \to \C^s \times \C^{|\Lambda|}$ as follows.	Let the $\mathcal{P}$-component of $\widetilde{\mathcal{P}}$ be given by the the period map $\mathcal{P}\colon \mathcal{A}^r(S)^{2n} \to \C^s$ with components $\mathcal{P}_i$:
	$$
	\mathcal{P}_i\colon 
	\mathcal{A}^r(S)^{2n} \ni (\psi_1,\ldots,\psi_n,\nu_1,\ldots,\nu_n) = (\psi,\nu) \mapsto \mathcal{P}_i(\psi,\nu) = 
	\sum_{j=1}^n\int_{C_i}\psi_j \, \dd\nu_j.
	$$
	Let the map $\mathcal{Z}\colon  \mathcal{A}^r(S)^{2n} \to \C^{|\Lambda|}$ be given in components $\mathcal{Z}_p \colon  \mathcal{A}^r(S)^{2n} \to \C$ by:
	$$
	\mathcal{Z}_p(\psi, \eta) = z'(p) - z'(p_0) + \sum_{j=1}^n\int_{E_p}\psi_j \,\dd\nu_j, \quad p \in \Lambda,
	$$
	The map
	\begin{equation*}
		\mathcal{S}\colon \label{xspray}
		(\zeta, \xi) \mapsto \tilde{\mathcal{P}}(\tilde{\bx}(\cdot,\zeta,\xi), \by)
	\end{equation*}
	is affine and its derivative at $(\zeta,\xi) = (0,0)$ is the identity map since $\partial_\zeta \mathcal{P}(\tilde{\bx}(\cdot,0,0), \by) = \id_s$ by \ref{gom=0}, $\partial _\xi \mathcal{P}(\tilde{\bx}(\cdot,0,0), \by) = 0$ by \ref{hom=0}, $\partial_\zeta \mathcal{Z}(\tilde{\bx}(\cdot,0,0), \by)) = 0$ by \ref{gz=0} and $\partial_\xi \mathcal{Z}(\tilde{\bx}(\cdot,0,0), \by) = \id_{|\Lambda|}$ by \ref{hz=1}. Depending on how well the approximations in Step 1 are chosen the value $\mathcal{S}(0,0)$ lies arbitrarily close to the origin, hence, there exists a unique point $(\zeta_0,\xi_0) \in \C^s \times \C^{|\Lambda|}$ close to the origin satisfying $\mathcal{S}(\zeta_0,\xi_0) = 0$. The map

	\begin{equation*}
		\mathcal{S}^\delta\colon  \label{x''spray}
		(\zeta, \xi) \mapsto \tilde{\mathcal{P}}(\tilde{\bx}^\delta(\cdot,\zeta,\xi), \by),
	\end{equation*}
	depending holomorphically on $\delta \in \C$, is still affine, with $\mathcal{S}^0 = \mathcal{S}$. By the implicit function theorem the point $(\zeta_0,\xi_0)$ depends holomorphically on $\delta$, thus it can be also made to lie arbitrarily close to the origin by choosing the parameter $\delta$ close enough to zero.
	
	The vanishing of the $\mathcal P$-component at $(\zeta_0,\xi_0)$ implies that the holomorphic 1-form
	$$
	\tilde{x}^\delta_1(\zeta_0,\xi_0)\,\dd y_1 + x_2\,\dd y_2 + \cdots + x_n\,\dd y_n
	$$
		is exact and hence integrates to the $z$-component of the Legendrian curve $F\colon \mathcal{R} \to \C^{2n+1}$ approximating $f$. The map $\mathcal{Z}$ also being zero at $(\zeta_0,\xi_0)$ then ensures the condition $z'(p)=z(p)$ at all points $p \in \Lambda$. By construction, the jets of $F$ then match the jets of $f$ at the points $p \in \Lambda$ to order at least $m(p)$. If the last condition in the theorem holds, one may choose $\delta \neq 0$ making $F$ an immersion.
	
	Finally, we explain how to fix a particular component function. Suppose one wishes to keep either $x_j'$ or $y_j'$ for some $j$. By Remark \ref{cntMf} we may assume this component is in fact $y_2'$ and one may just define $y_2 = y_2'$ instead of approximating in Step 1, thus fixing this particular component. Note that $y_2'$ is chosen instead of $y_1'$ since we require $y_1$ to be made nonconstant while posing no restrictions on $y_2$.
	
	It rests to prove one may also fix the $z'$-component. In Step 1, put $z=z'$ and assume $x_1'$ is nonconstant (if it is, approximate it by using Mergelyan's theorem with a nonconstant function having the correct $m(p)$-jets at every point $p \in \Lambda$). The set of zeroes $Z_x$ of $x_1'$ is thus finite in $R$. Now, choose a holomorphic function $h\colon \r \to \C$ which has the same zeroes as $x_1$ in $R$ and such that those zeroes are of the same order. By Mergelyan's theorem there exists a holomorphic function $g\colon  \r \to \C^*$ approximating $x_1'/h$ and agreeing with $x_1'/h$ to order at least $m(p)$ at every point $p \in \Lambda$, thus the holomorphic function $x_1 := hg$ approximates well the function $x_1'$, agrees with $x_1'$ to order at least $m(p)$ at every point $p \in \Lambda$, and has the same zeroes as $x_1'$ in $R$ which are of the same order as those of $x_1'$. Moreover, if $m(p)=0$ and $x_1'(p)=0$ for some $p \in \Lambda$, we require that $x_1$ has nonvanishing derivative at $p$.
	
	Approximate $x_i',y_i'$ for $i \neq 2$ with nonconstant entire holomorphic functions $x_i,y_i\colon \r \to \C$ such that $x_i,y_i$ agree with $x_i',y_i'$ respectively to order at least $m(p)$ at every point $p \in \Lambda$. Furthermore, we may ensure the holomorphic 1-form $\beta' = - \dd z - \sum_{i\neq 2}x_i\dd y_i$ is equal to zero at $p \in R$ whenever $x_1(p) =0$ and the order of the zero of $\beta'$ at $p$ is at least the order of the zero of $x_1$ at $p$, hence the 1-form $\beta = \beta'/x_1$ is holomorphic on a neighbourhood of $R$.
	
	In order to make the approximating map an immersion, we construct the function $\eta$ similarly as in Step 1, but now we make sure the map $(x_1e^{\delta\eta}, y_1)\colon  \r \to \C^2$ has nonvanishing derivative on a neighbourhood of $R$ for any $\delta \in \C^*$ instead. Note that for $p \in \Lambda$ with $m(p) =0$, the function $x_1$ has nonvanishing derivative at $p$, hence so does $x_1e^{\delta \eta}$ by requiring $\dot{\eta}(p) = 0$.
	
	In this case, Step 2 is the same as above, thus proceed to Step 3. We similarly construct the functions $g_j$, but we now require that
	\[
	\int_{C_i} g_j \beta = \delta_{ij} \quad
	\text{and} \quad
	\int_{E_p} g_j \beta = 0.
	\]
	Of the functions $h_p$ in Step 4 we require that
	\[
	\int_{C_i} h_p \beta = 0 \quad
	\text{and} \quad
	\int_{E_p} h_q \beta = \delta_{pq}.
	\]
	The spray we construct in Step 5 in this case is of the form
	\[
	\tilde{x}_1(u,\zeta,\xi) = x_1(u) \exp \left(\sum_{i=1}^s\zeta_j g_j(u) + \sum_{p \in \Lambda}\xi_p h_p(u) + \delta \eta(u)\right).
	\]
	Define $\tilde{\beta}(u,\zeta,\xi) = \beta'(u)/\tilde{x}_1(u,\zeta,\xi)$. The $\p_i$ and $\mathcal{Y}_p$ components of the map
	\[\mathcal{S}\colon  \C^s \times \C^{|\Lambda|} \to \C^s \times \C^{|\Lambda|}\]
	are given by
	\[
	\p_i(\zeta,\xi) = \int_{C_i} \tilde{\beta}(\cdot,\zeta,\xi), \quad
	\mathcal{Y}_p(\zeta,\xi) = y_1'(p_0)- y_1'(p) + \int_{E_p} \tilde{\beta}(\cdot,\zeta,\xi).
	\]
	The same argument as before shows that $\tilde{\beta}(\cdot,\zeta_0,\xi_0)$ is exact  and that the $\mathcal{Y}_p$ components vanish for some $(\zeta_0,\xi_0)$. We have thus obtained a Legendrian immersion by integrating $\tilde{\beta}$ to the function $y_1$ on a neighbourhood of $R$.

\end{proof}

%
%





The next lemma shows that one may still use Lemma \ref{finiteInterpolation} in order to prescribe the jets to the approximating curve even though the points in $\Lambda$ may lie outside of the domain of the approximated Legendrian curve:

\begin{lm}
	\label{outsideInterpolation}
	Let $S=K \cup \Gamma \subset \r$, $f\colon S \to \C^{2n+1}$, $R \subset \r$, be as in Lemma \ref{finiteInterpolation}. Suppose we are given
	\begin{enumerate}[label=(\roman*)]
		\item a finite set $\Lambda \subset \Int R$, not intersecting $bK \cup \Gamma$, $\Lambda' = \Lambda \cap \Int S$, $\Lambda'' = \Lambda \cap \Int R \backslash S$ (note that $\Lambda = \Lambda' \cup \Lambda''$);		
		\item a holomorphic Legendrian curve $\phi\colon  O \to \C^{2n+1}$, defined on an open neighbourhood $O \subset \Int R$ of the set $\Lambda''$ and
		\item a function $m\colon  \Lambda \to \N \cup \{0\}$.
	\end{enumerate}
	Then there exists a generalised Legendrian curve $\tilde{f}\colon  S' \to \C^{2n+1}$ of class $\a^r(S')$ defined on some Runge admissible set $S' \subset \Int R$, which contains $S$ in its interior and is a deformation retract of $R$, such that $\tilde{f}$ agrees with $f$ to order $m(p)$ at $p$ for all $p \in \Lambda'$ and $\tilde{f}$ agrees with $\phi$ to order $m(p)$ for all $p \in \Lambda''$. If for some $1\leq j \leq 2n+1$ the projection $f_j$ of $f$ to the $j$-th coordinate in $\C^{2n+1}$ is already holomorphic on a neighbourhood of $R$ and agrees with $\phi$ to order at least $m(p)$ at every point $p \in \Lambda''$, it may be kept fixed so that $\tilde{f}_j = f_j$.
\end{lm}

\begin{proof}
	As in the previous proof we assume $R$ and hence $S$ is connected. By Lemma \ref{finiteInterpolation}, there exists a holomorphic Legendrian immersion $\tilde{f}\colon  \r' \to \C^{2n+1}$, defined on some open neighbourhood $\r'$ of $S$, approximating $f$ in $\c^r(S)$-topology and interpolating $f$ at the points $p \in \Lambda'$ to order at least $m(p)$. Choose a smoothly bounded relatively compact neighbourhood $U$ of the set $S$ in $\r$ such that $U$ contains no point in $\Lambda''$ and such that $U$ deformation retracts onto $S$. Then $\overline{U}$ is a compact Runge admissible subset in $R$. Now, for all $p \in \Lambda''$ choose embedded holomorphic disks $\Omega_p$, $\overline{\Omega}_p \subset O_p$, containing $p$ such that
	\begin{enumerate}[label=\alph*)]
		\item $\overline{\Omega}_p \subset \Int R \backslash \overline{U}$,
		\item $\overline{\Omega}_p \cap \overline{\Omega}_q = \emptyset$ for all $p \neq q$.
	\end{enumerate}
	For each $p \in \Lambda''$ choose a smooth embedded Jordan arc $E_p$ connecting some point in $b\overline{U}$ to $p$ such that
	\begin{enumerate}[label = \alph*)]
		\setcounter{enumi}{2}
		\item $E_p \cap E_q = \emptyset$, i.\ e.\ the arcs $E_p$ are pairwise disjoint,
		\item $E_p$ and $b\overline{U}$ intersect transversely in a single point, which we label $u_p$; 
		\item $E_p$ and $b\Omega_p$ intersect transversely in a single point, which we label $\omega_p$;

		\item $(\overline{\Omega}_p \cup E_p) \cap (\overline{\Omega}_q \cup E_q) = \emptyset$ for any distinct points $p,q \in \Lambda$.
	\end{enumerate}

	Let $\Omega =\cup_{p\in \Lambda''}\Omega_p$ and $E = \cup_{p \in \Lambda''}E_p$. It is immediate that $S'= \overline{U} \cup \Omega \cup E$ deformation retracts onto $S$ and is thus Runge in $\Int R$. By swapping $E_p$ for the subarc $E_p'$ of $E_p$, connecting the points $u_p$ and $\omega_p$, we obtain the admissible set $S' = \overline{U} \cup \bigcup_{p\in\Lambda''}(E_p' \cup \overline{\Omega}_p)$. Furthermore, all points $p \in \Lambda = \Lambda' \cup \Lambda''$ lie in $\Int S'$. 
	
	Denote $\phi$ in coordinates by $\phi = (\bx',\by', z') \in \C^{2n+1}$. Extend the component functions $x_i,y_i$ (we use the same notation, i.\ e.\ $x_i$ and $y_i$, to denote these extensions) of $\tilde{f}$ to $\c^r$-smooth functions on $S'$, holomorphic on $\Int S' = \Int K \cup \Omega$, requiring
	\begin{enumerate}[label=\Alph*)]
		\item 
		\label{xeq}
		$\bx|_{\overline{\Omega}_p} = \bx'$ for all $p \in \Lambda''$,
		\item
		\label{yeq}
		$\by|_{\overline{\Omega}_p} = \by'$ for all $p \in \Lambda''$ and
		\item
		\label{zeq}
		$z'(p)+\sum_{i=1}^n\int_{E_p}x_i\dd y_i = z(u_p)$ for all $p \in \Lambda''$.
	\end{enumerate}
	Such maps are obtained by first fixing them on the pairwise disjoint sets $\overline{\Omega}_p$ and $\overline{U}$ and then smoothly extending them over the arcs $E_p$ with smooth paths which agree with $\bx'$ and $\by'$ to order $r$ at $u_p$ and $\omega_p$. Choose a point $u_0 \in \Int K$ and extend the $z$-component to $S'$ by setting
	\begin{equation}
		\label{zint}
		z(u) = z(u_0) - \sum_{i=1}^n\int_{u_0}^u x_i\dd y_i,
	\end{equation}
	so that the extension of $\tilde{f}$ to $S'$ remains Legendrian. Since $S$ is a deformation retract of $S'$, the above integral is well-defined. By construction, $S'$ is a Runge admissible set in $\r$, containing $S$ in its interior, and $\tilde{f}$ is the required generalised Legendrian curve on $S'$.
	
	If some $x_k$ or $y_k$ is already holomorphic on a neighbourhood of $R$, replace \ref{xeq} or \ref{yeq} by $x_i|_{\overline{\Omega}_p} = x_i'$ for all $i \neq k$ or $y_i|_{\overline{\Omega}_p} = y_i'$ for all $i \neq k$ respectively, construct $x_i'$ and $y_i'$ over the arcs $E_p$ such that \ref{zeq} holds and obtain the $z$-component by integrating as in \ref{zint}.
	
	Suppose now $z$ is holomorphic on a neighbourhood of $R$. For $i \geq 2$ and every $p \in \Lambda''$ we extend the component functions $x_i$, $y_i$ to $\Omega_p$ by setting $x_i|_{\Omega_p} = x_i'$ and $y_i|_{\Omega_p} = y_i'$. Denote by $\beta$ the holomorphic 1-form $\beta = -\dd z - \sum_{i=2}^n x_i \dd x_i$ and note that if $\beta(p) = 0$ for some $p \in \Lambda''$ such that $m(p) \geq 1$, then $x_1'(p) = 0$ by the Legendrian condition $\phi^*\alpha=0$. Choose a holomorphic function $g\colon  \Omega \to \C$ such that 
	\begin{enumerate}[label=G$_\arabic*$:]
		\item $g$ agrees with $\phi$ to order at least $m(p)$ for all $p \in \Lambda''$,
		\item if $\beta(p) = 0$ for some $p \in \Lambda''$, then $g(p) = 0$ and the order of this zero of $g$ is no more than the order of the zero $p$ of $\beta$,
		\item after shrinking $\Omega_p$ if necessary, $g$ has no other zeroes in $\Omega_p$ for all $p \in \Lambda''$.
	\end{enumerate}
	Note that with these properties, the 1-form $\tilde{\beta} = \beta / g$ is holomorphic in $\Omega$. Extend the $x_1$ component of $f$ to $\Omega$ by $x_1|_{\Omega} = g$. Now, extend $x_i,y_i$ for $i \geq 2$, $z$ and $x_1$ over the arcs $E_p'$ to functions of class $\a^r(S')$ such that $x_1$ has no zeroes on $E_p'$ (hence it also has no zeroes on $E_p \subset E_p' \cup \Omega_p$) and $\int_{E_p} \beta/x_1 = \phi(p) - y_1(u_p)$ holds for every $p \in \Lambda''$. For a fixed $u_0 \in \Int K$ define $y_1(u) = y_1(u_0) + \int_{u_0}^u \beta/x_1$, note that the integral is path independent, and observe, that the curve $f=(x_1,y_1,\ldots,x_n,y_n,z)$ is Legendrian thus completing the proof.
\end{proof}

\subsection{Approximation by embeddings}
\label{cptbrd}

A \emph{compact bordered Riemann surface} $M$ is the closure of a smoothly bounded relatively compact open subset $M \subset \mathcal{R}$ in an open Riemann surface $\mathcal{R}$ whose boundary $bM$ consists of finitely many (smooth) Jordan curves. In this section we show holomorphic Legendrian curves in $\C^{2n+1}$ defined on compact bordered Riemann surfaces may be approximated uniformly in $\c^r$-topology with Legendrian embeddings, interpolating at a finite number of points to any given finite order. The result will be exploited in the proof of theorems \ref{mergelyan} and \ref{carleman} in order to make the approximating map a limit of embeddings thus guaranteeing its injectivity.

Once more we apply the methods from \cite{Alarcon2017}, adapting the proof of \cite[Lemma 4.4]{Alarcon2017} in order to include interpolation. The idea is to construct a family of holomorphic Legendrian immersions
$$
H(\cdot, \xi) \colon  M \to \C^{2n+1}, \quad \xi \in \C^N
$$
where $M$ is a compact bordered Riemann surface, such that the \emph{difference map}
$$
\delta H \colon  M\times M \times \C^N \to \C^{2n+1}
$$
given by $\delta H (u,v,\xi) = H(v,\xi) - H(u,\xi)$ is a \emph{submersive family of maps} on $M \times M \backslash U'$, where $U'$ is a suitably chosen open set in $M \times M$, containing the diagonal $D_M$ in $M \times M$, meaning the differential $\partial_\xi (\delta H)$ at $\xi = 0$ is surjective along $M \times M \backslash U'$. By Sard's theorem, $\delta H(\cdot,\cdot,\xi)$ is then transverse to the origin $0 \in \C^{2n+1}$, viewed as a submanifold in $\C^{2n+1}$, for a generic choice of parameter $\xi$ near 0. Since the image of $M \times M \backslash U'$ by the map $\delta H(\cdot,\cdot,\xi)$ is (at most) two-dimensional and the origin is zero dimensional, transversality implies the two submanifolds do not meet, thus $H(\cdot,\xi)$ is injective. Since $M$ is compact, it is an embedding.

\begin{thm}\label{borderedInterpolation}
	Suppose $M$ is a compact bordered Riemann surface and $f\colon  M \to \C^{2n+1}$ is a holomorphic Legendrian curve of class $\a^r(M)$ in $\C^{2n+1}$ with the standard contact structure. For a given finite set $\Lambda \subset \Int M$ such that the values $f(p)$ for $p \in \Lambda$ are pairwise distinct (i.\ e.\ $f|_\Lambda$ is injective) and a function $m\colon  \Lambda \to \N\cup \{0\}$ such that $\dd f(p) \neq 0$ for all $p \in \Lambda\cap \{m\geq 1\}$, there exists a holomorphic Legendrian embedding $F\colon  M \hookrightarrow \C^{2n+1}$ arbitrarily close to $f$ in the $\c^r(M)$-topology such that the map $F-f$ has a zero of order at least $m(p)$ at every point $p \in \Lambda$.
\end{thm}

\begin{proof}
	We prove the theorem in the case $n=1$. It will be clear from the proof that the result may be extended to arbitrary $n \in \N$. Since $M$ is a compact bordered Riemann surface, it may be viewed as a compact Runge admissible set in some open Riemann surface $\r$. By Lemma \ref{finiteInterpolation}, we may assume $f$ is a holomorphic Legendrian immersion $f=(x,y,z) \colon  \r \to \C^3$. By assumption, there exist open disks $\Omega_p$ for $p \in \Lambda$ such that $f|_{\Omega_p}$ is an embedding for all $p$. Moreover, by shrinking the sets $\Omega_p$ if necessary, we may assume the difference map $\delta f\colon  M \times M \to \C^3$, given by $\delta f(u,v) = f(v) - f(u)$ is non-vanishing on the open set
	$$
	\Omega = \bigcup_{(p,q) \in \Lambda \times \Lambda \backslash D_M} \Omega_p \times \Omega_q,
	$$
	where $D_M$ is the diagonal in $M \times M$. Since $f$ is an immersion, there is an open neighbourhood $U$ of $D_M$ such that $\delta f$ is non-vanishing also on $U \backslash D_M$. We now put $U' = U \cup \Omega$. 
	
	Fix $(u,v) \in M \times M \backslash U'$. By \cite[Lemma 1.12.10]{Alarcon2021}, there exists a connected Runge subset $C'' \subset M$, consisting of finitely many smooth Jordan curves $C_1,\ldots, C_l$ such that $C_i \cap C_j = \{u_0\}$ for some fixed point $u_0 \in \Int M\backslash \Lambda$ and all $i\neq j$. By slightly deforming these curves if necessary, we may assume $C''$ does not intersect the discrete set $\Lambda \cup \{\dd y = 0\}$. Connect the points $u,v \in M$ with a smooth Jordan arc $E$, disjoint from $\Lambda$, such that $E$ intersects $C''$ only at $u_0$ or not at all and the union $C' \cup E$ is Runge in $M$. 
	
	 Choose a small closed disk $D$ centered at $u_0$ which does not intersect $\Lambda$ such that $E \backslash D \neq \emptyset$ and $C_i \backslash D \neq \emptyset$ for all $i=1,\ldots,l$ and denote by $C'$ the union $C'= C'' \cup E \cup D$. Now for each $p \in \Lambda$ choose a small embedded disk $O_p \subset \r$, $p \in \Int O_p$, not intersecting $C'$ and connect $u_0$ to $p$ with a smooth Jordan arc $E_p$ such that the intersection $E_p \cap O_p$ is transverse and consists of a single point labeled $o_p$ and the intersection $E_p \cap bD$ is also transverse and consists of a single point $d_p \in bD$. Now label the subarc of $E_p$ connecting $d_p$ to $o_p$ by $E_p'$. By slightly deforming the arcs $E_p'$ if necessary we may ensure the sets $O_p \cup E_p'$ and $O_q \cup E_q'$ do not intersect when $p \neq q$. Denote by $C$ the union $C = C' \cup \bigcup_{p\in \Lambda}(O_p \cup E_p')$. Note that $C$ deformation retracts onto $C'$ and is thus a Runge admissible subset in the open Riemann surface $\Int M$.
	 
	 Using the same method as in the proof of Lemma \ref{finiteInterpolation} we construct holomorphic functions $g_i\colon \r \to \C$ for $i=1,\ldots,l$ such that
	\begin{enumerate}[label=G$_{\arabic*}$:, ref=G$_{\arabic*}$]
		\item \label{gzero} $g_i \approx 0$ on $D\cup E$ for all $i=1,\ldots,l$,
		\item \label{gdelta} $\int_{C_j}g_i \dd y \approx \delta_{ij}$ for all $i,j= 1,\ldots, l$,
		\item \label{gjet} $g_i^{(k)}(p) = 0$ for all $k = 0,\ldots,m(p)$ and all $p \in \Lambda$,
		\item \label{gimm} $\dd g_i(p) = 0$ for all $p \in \r$ such that $\dd y(p) = 0$,
		\item \label{gint} $\int_{E_p} g_i \dd y \approx 0$ for all $i=1,\ldots,l$ and all $p \in \Lambda$,
	\end{enumerate}
	We then construct holomorphic functions $h_p\colon \r \to \C$ for $p \in \Lambda$ such that
	\begin{enumerate}[label=H$_{\arabic*}$:, ref=H$_{\arabic*}$]
		\item \label{hzero} $h_p \approx 0$ on $C'$,
		\item \label{hjet} $h_p^{(k)}(q) = 0$ for all $k=0,\ldots,m(p)$ and all $p,q \in \Lambda$,
		\item \label{himm} $\dd h_p(q) = 0$ for all $q \in \r$ such that $\dd y(q) = 0$
		\item \label{hdelta} $\int_{E_q}h_p\dd y \approx \delta_{pq}$ for all $p,q \in \Lambda$,
		\item \label{hpath} $|h_p(q)|< 1$ for all $q \in E$ and $p \in \Lambda$.
	\end{enumerate}
	Next, construct holomorphic functions $w_1,w_2 \colon  \r \to \C$ such that
	\begin{enumerate}[label=W$_{\arabic*}$:, ref=W$_{\arabic*}$]
		\item \label{wzero} $w_j \approx 0$ on $C'' \cup D$ for $j=1,2$,
		\item \label{wjet} $w_j^{(k)}(p)=0$ for all $k=0,\ldots,m(p)$ and all $p \in \Lambda$, $j=1,2$.
		\item \label{wpath} $w_1(u) = 0$, $w_1(v) =1$, $w_2(u) = w_2(v)=0$,
		\item \label{w1small} $|w_1(p)| \leq 1$ for all $p \in E$,
		\item \label{wint} $\big| 1+\int_E w_2\dd y \big| < \mu$,
		\item \label{wsmall} $|w_j(p)| < \mu$ and $|\dd w_j(p)| < \mu$ for all $p \in C$, $j=1,2$,

	\end{enumerate}
	
	Again we make use of the extended period map $(\p, \mathcal{Z})\colon \a^r(M)^2 \to \C^l \times \C^{|\Lambda|}$, where \linebreak $\p = (\p_k)_{k=1}^l$,  $\mathcal{Z} = (\mathcal{Z}_p)_{p \in \Lambda}$, and $\p_k, \mathcal{Z}_p \colon  \a^r(M)^2 \to \C$ are given by
	\[
	\p_k(\phi,\rho) = \int_{C_k}  \phi \, \dd \rho, \quad
	\mathcal{Z}_p(\phi,\rho) = z(p) - z(u_0) + \int_{E_p} \phi \, \dd \rho.
	\]
	Let $\xi = (\xi_1,\xi_2,\xi_3) \in \C^3$, $\zeta = (\zeta_i)_{i=1}^l \in \C^l$, $\varsigma = (\varsigma_p)_{p\in \Lambda} \in \C^{|\Lambda|}$, define
	\[
	\tilde{x}(\cdot, \xi_1,\xi_3,\zeta, \varsigma) = x + \xi_1w_1 + \xi_3 w_2 + \sum_{i=1}^{l}\zeta_i g_i + \sum_{p\in \Lambda}\varsigma_p h_p, \quad
	\tilde{y}(\cdot, \xi_2) = y + \xi_2 w_1,
	\]
	and denote by $\f\colon  \C^3 \times \C^{l}\times \C^{|\Lambda|} \to \a^r(M)^2$ the map
	\[
	\f\colon  \C^3 \times \C^{l}\times \C^{|\Lambda|} \ni (\xi,\zeta, \varsigma) \mapsto
		(\tilde{x}(\cdot, \xi_1,\xi_3,\zeta, \varsigma), \tilde{y}(\cdot, \xi_2)) \in \a^r(M)^2.
	\]
	The composition $ \tilde{\mathcal{F}}=(\p,\z) \circ \mathcal{F}$ is then a holomorphic map
	$\tilde{\mathcal{F}} \colon  \C^3_\xi \times \C^l_\zeta \times \C^L_\varsigma \to \C^l \times \C^L$ with components $\tilde{\p}_k := \p_k \circ \f$, $k=1,\ldots,l$ and $\tilde{\z}_p = \mathcal{Z}_p \circ \f$ for $p\in \Lambda$.	By property \ref{gdelta}, the matrix of partial derivatives
	$$
	\partial_\zeta \tilde{\p}\big|_{(\xi,\zeta,\varsigma)=0} = \left(\frac{\partial \tilde{\p}_k}{\partial \zeta_j}\right)_{j,k=1}^l \bigg|_{(\xi,\zeta,\varsigma)=0}
	$$
	is approximately the identity matrix. Same is true for
	$$
	\partial_\varsigma \tilde{\mathcal{Z}} \big|_{(\xi,\zeta,\varsigma)=0}=
	\left(
	\frac{\partial \tilde{\mathcal{Z}}_p}
	{\partial \varsigma_q}
	\right)_{p,q\in \Lambda}\bigg|_{(\xi,\zeta,\varsigma)=0}
	$$
	by property \ref{hdelta}. By \ref{gint}, the matrix $\partial_\zeta \tilde{\mathcal{Z}}|_{(\xi,\zeta,\varsigma)=0}$ is approximately the zero matrix and by \ref{hzero} the matrix $\partial_\varsigma \tilde{\p}|_{(\xi,\zeta,\varsigma)=0}$ is approximately the zero matrix as well thus the matrix of partial derivatives
	$$
	\partial_{\zeta,\varsigma} \tilde{\f}|_{(\xi,\zeta,\varsigma)=0} =
	\begin{bmatrix}
		\partial_\zeta \p & \partial_\varsigma \p \\
		\partial_\zeta \mathcal{Z} & \partial_\varsigma \mathcal{Z}
	\end{bmatrix}\bigg|_{(\xi,\zeta,\varsigma)=0}
	\approx
	\begin{bmatrix}
		\id_l & 0 \\
		0 & \id_L
	\end{bmatrix}
	$$
	is invertible and by the implicit function theorem we can solve the equation $\tilde{\f}(\xi,\zeta,\varsigma) = 0$ on $\xi$, that is express the variables $\zeta=\zeta(\xi)$ and $\varsigma=\varsigma(\xi)$ as holomorphic functions of $\xi$, so that $\tilde{\f}(\xi,\zeta(\xi),\varsigma(\xi)) = 0$ for all $\xi \in \C^3$ sufficiently close to the origin. For a fixed such $\xi$, $\tilde{\p}$-components being zero imply the holomorphic 1-form $\tilde{x}(\xi) \dd \tilde{y}(\xi)$ is exact, hence it integrates to the $\tilde{z}$-component 
	\[
	\tilde{z}(t,\xi) = z(u_0) - \int_{u_0}^t \tilde{x}(s,\xi) \dd \tilde{y}(s,\xi)
	\]
	of a holomorphic Legendrian curve $H(\cdot,\xi) = H^{(u,v)}(\cdot,\xi)\colon  M \to \C^3$, given by 
	\[
	H(t,\xi) = (\tilde{x}(t,\xi,\zeta(\xi),\varsigma(\xi)), \tilde{y}(t,\xi),\tilde{z}(t,\xi)).
	\]
	Since also the $\tilde{\z}$-components of the map $\tilde{\f}$ are zero, it follows that $\tilde{z}(p,\xi) = z(p)$ for all $p \in \Lambda$ and all sufficiently small $\xi \in \C^3$. Hence, by construction, the curve $t \mapsto H(t,\xi)$ satisfies the interpolation condition from the theorem. Moreover, since $H(\cdot,0) = (x,y)$ is an immersion, so is $H(\cdot,\xi)$ for sufficiently small $\xi$ by the Cauchy estimates since $M$ is compact. 
	
	We shall now prove that the differential $\partial_\xi \delta H(u,v, \xi)|_{\xi = 0}$ is surjective. From \[\tilde{\f}(\xi,\zeta(\xi),\varsigma(\xi)) = 0\] we obtain by the chain rule:
	\begin{align*}
		\frac{\partial}{\partial \xi}
		\bigg|_{\xi=0} \tilde{\p} 
		+ \left(\frac{\partial}{\partial \zeta}
		\bigg|_{\zeta=0}\tilde{\p}, \; \frac{\partial}{\partial \varsigma}
		\bigg|_{\varsigma = 0} \tilde{\p}\right)^\top
		\left(
		\pdv{\zeta}{\xi}\bigg|_{\xi=0} ,
		\pdv{\varsigma}{\xi}\bigg|_{\xi=0} 
		\right)
		&= 0 \\
		\frac{\partial}{\partial \xi}
		\bigg|_{\xi=0} 
		\tilde{\mathcal{Z}} +
		\left(
		\frac{\partial}{\partial \zeta}
		\bigg|_{\zeta=0}\tilde{\mathcal{Z}}, \; \frac{\partial}{\partial \varsigma}
		\bigg|_{\varsigma = 0} \tilde{\mathcal{Z}}\right)^\top
		\left(
		\pdv{\zeta}{\xi}\bigg|_{\xi=0} ,
		\pdv{\varsigma}{\xi}\bigg|_{\xi=0} 
		\right)
		&= 0
	\end{align*}
	By \ref{wsmall}, $\partial_\xi\tilde{\p}|_{\xi=0}$ is of size $O(\mu)$ and by \ref{wjet}, also $\partial_\xi \tilde{\z}|_{\xi=0}$ is of size $O(\mu)$. Since the above matrix of $\partial_\zeta$ and $\partial_\varsigma$ partial derivatives of the extended period map $(\tilde{\p},\tilde{\z})$ is approximately the identity matrix, we conclude that also the matrix of derivatives $(\partial_\xi \zeta, \partial_\xi \varsigma)|_{\xi = 0}$ is of size $O(\mu)$. We now compute the differential of the difference map $\delta H(\cdot, \cdot,\xi)$ at $(u,v) \in M \times M \backslash U'$ and the parameter value $\xi = 0$. Its first component is
	\begin{align*}
		\frac{\partial}{\partial \xi}
		\bigg|_{\xi=0}
		\delta\tilde{x}
		(u,v,\xi,\zeta(\xi),\varsigma(\xi)) 
		&=  
		(
		\partial_\xi\tilde{x}(v,0) + 
		\partial_\zeta 
		\tilde{x}(v,0)\partial_\xi\zeta(0) + \partial_\varsigma \tilde{x}(v,0)\partial_\xi\varsigma(0)) - \\
		&-(
		\partial_\xi\tilde{x}(u,0) + 
		\partial_\zeta \tilde{x}(u,0)\partial_\xi\zeta(0) + 
		\partial_\varsigma \tilde{x}(u,0)\partial_\xi\varsigma(0))
	\end{align*}
	Now, $\partial_\xi\tilde{x}(v,0) = (1,0,0)$ and $\partial_\xi\tilde{x}(u,0) = 0$ by \ref{wpath}. Moreover, the derivatives $\partial_\zeta \tilde{x}$ and $\partial_\varsigma \tilde{x}$ at $(\xi,\zeta,\varsigma) = 0$ are indepentent of $\mu$, thus all the other terms in the above sum add up to a term of the size $O(\mu)$, since $\partial_\xi\zeta(0)$ and $\partial_\xi\varsigma(0)$ are of the size $O(\mu)$. It follows that
	\[
	\frac{\partial}{\partial \xi}\bigg|_{\xi=0}\delta\tilde{x}(u,v,\xi,\zeta(\xi),\varsigma(\xi)) = (1,0,0) + O(\mu).
	\]
	Similarly we obtain
	\[
	\frac{\partial}{\partial \xi}\bigg|_{\xi=0} \tilde{y}(u,v,\xi) = \partial_\xi \tilde{y}(v,0) - \partial_\xi \tilde{y}(u,0) = (0,w_1(v) - w_1(u),0) = (0,1,0).
	\]
	To estimate the $\xi$-derivative of the $\tilde{z}$ component, first observe that by \ref{gzero}, \ref{hzero} and \ref{wsmall} we have
	\begin{equation}
		\label{pdvx}
	\pdv{\xi_j} \int \tilde{x} d\tilde{y} = \int (\partial_{\xi_j} \tilde{x}) \dd \tilde{y} + \int \tilde{x} \dd (\partial_{\xi_j} \tilde{y}).
	\end{equation}
	The last term in the above sum is nonzero only when $j=2$ and in that case we have $\partial_{\xi_2} \tilde{y} =  w_1$. We compute
	\begin{multline*}
	-
	\left(
	\int_{u_0}^u \tilde{x} \dd (\partial_{\xi_2}\tilde{y}) - \int_{u_0}^v \tilde{x} \dd (\partial_{\xi_2}\tilde{y})
	\right) =
	\int_E \tilde{x} \dd (\partial_{\xi_2}\tilde{y}) = \\
	=
	\int_E x\, \dd w_1 + \xi_1 \int_E w_1 \dd w_1 + \xi_3 \int_E w_2 \dd w_1
	 + \sum_{i=1}^l\zeta_i \int_E g_i \dd w_1 + \sum_{p\in \Lambda} \varsigma_p \int_E h_p \dd w_1,
	\end{multline*}
	and observe that by \ref{wpath} the term $\int_E x \, \dd w_1$ is bounded in absolute value by some constant, independent of $\mu$, whereas all the other terms vanish at $\xi=0$. In the first term of equation \ref{pdvx}, we obtain:
	\[
	\partial_{\xi_1} \tilde{x} = w_1 + \sum_{i=1}^l (\partial_{\xi_1}\zeta_i) g_i + \sum_{p\in \Lambda} (\partial_{\xi_1}\varsigma_p)h_p,
	\]
	thus $\int_{u_0}^t (\partial_{\xi_1} \tilde{x})\dd \tilde{y}$ is bounded from above in absolute value by $C + O(\mu)$ for $t \in \{u,v\}$ where $C$ is some constant, independent of $\mu$, by \ref{w1small}, \ref{gzero}, \ref{hzero} and the fact that $\partial_\xi \zeta_i$ and $\partial_\xi \varsigma_p$ are of the size $O(\mu)$. On the other hand, $\partial_{\xi_2}\tilde{x} = 0$ by construction. Finally, we compute
	\begin{align*}
		\frac{\partial}{\partial \xi_3}
		\bigg|_{\xi=0}
		\delta\tilde{z}(u,v,\xi) 
		&= 
		\frac{\partial}{\partial \xi_3}
		\bigg|_{\xi=0}
		\left(
		\int_{u_0}^u \tilde{x}\dd\tilde{y} - \int_{u_0}^v\tilde{x}\dd\tilde{y}
		\right) = 
		-\frac{\partial}{\partial \xi_3}
		\bigg|_{\xi=0} 
		\int_E \tilde{x}\dd\tilde{y} = 										\\
		&= -\left(
		\int_E (\partial_{\xi_3}\tilde{x})\dd\tilde{y} 
		+ \tilde{x}(\partial_{\xi_3} \dd\tilde{y})
		\right)
		\bigg|_{\xi=0} = 													\\
		&= 
		-\int_E w_2\dd y - 
		\int_E  
		\left(
		\langle 
		(\partial_{\xi_3}\zeta)(0),
		(g_i)_{i=1}^l
		\rangle + 
		\langle(
		\partial_{\xi_3} \varsigma)(0),
		(h_i)_{i=1}^L \rangle\right)\dd y=									\\
		&= 1 + O(\mu),
	\end{align*}
	where in the last equality we have used property \ref{wint}. Thus, the partial differential $\partial_\xi \delta H$ at $\xi = 0$ is of the form
	\[
	\partial_\xi (\delta H) \big|_{\xi=0} = 
	\begin{bmatrix}
		1 &0 &0 \\
		0 &1 &0 \\
		\ast & \ast & 1
	\end{bmatrix}	
	+ \, O(\mu),
	\]
	where the components of the above matrix denoted by $\ast$ are bounded from above in absolute value by constants, independent of $\mu$. It follows that for a small enough $\mu > 0$ the matrix is invertible thus $(\partial_\xi \delta H)|_{\xi=0}\colon  \C^3 \to \C^3$ is an isomorphism at the point $(u,v) \in M \times M \backslash U'$, i.\ e.\ $\delta H$ is a submersion at $(u,v)$. Since this is an open property, $\delta H = \delta H^{(u,v)}$ is a submersion on some small neighbourhood $W_{u,v}$ of $(u,v) \in M \times M \backslash U'$.
	
	We may repeat this construction near any point $(u,v) \in M\times M \backslash U'$ and obtain a family of maps $\{H^{(u,v)}\}_{(u,v) \in M \times M \backslash U'}$ and an open cover $\{W_{u,v}\}_{(u,v) \in M \times M \backslash U'}$ for $M \times M \backslash U'$ such that the difference map $\delta H^{(u,v)}$ is a submersion on $W_{u,v}$ for all $u,v$. Since $M \times M \backslash U'$ is compact, finitely many $W_{u,v}$ cover $M \times M \backslash U'$. We relabel these sets by $W_1, \ldots, W_N$ and the associated sprays by $H^1,\ldots,H^N$. We now compose these sprays in the following way: for each $j=1,\ldots,N$ put $\tilde{H}^j(t,\xi^j) = H^j(t,\xi^j) - f(t)$ where $\xi^j = (\xi^j_1, \xi^j_2, \xi^j_3) \in \C^3$ and define for $\xi = (\xi^1,\ldots,\xi^N) \in \C^{3N}$ and $t \in M$:
	\[
	H(t,\xi) = f(t) + \sum_{j=1}^N\tilde{H}^j(t,\xi^j).
	\]
	
	Any point $(u,v) \in M \times M \backslash U'$ is contained in some $W_j$, hence the partial derivative \linebreak $\partial_{\xi^j}(\delta H)(u,v,\xi)|_{\xi^j=0} = \partial_{\xi^j}(\delta\tilde{H})^j(u,v,\xi^j)|_{\xi^j=0}$ is an isomorphism at $(u,v)$ by construction. Let $H_0 = H(\cdot,\cdot,0)$. Since $(u,v)$ was arbitrary, $\delta H_0$ is submersive on $M \times M \backslash U'$, i.\ e.\ $\delta H_0$ is transverse to the submanifold $\{0\} \subset \C^3$ along $M \times M \backslash U'$.	By the parametric transversality theorem (see e.~g.\ \cite[Chapter 3, Theorem 2.7.]{Hirsch1976}), the difference map $\delta H(\cdot,\cdot,\xi)$ is transverse to the submanifold $\{0\} \in \C^3$ for a generic choice of $\xi = (\xi^1,\ldots,\xi^N) \in \C^{3N}$, hence, by dimension reasons, the difference map omits the value 0, meaning that $H(\cdot,\xi)$ is injective on $M \times M \backslash U'$.
	
	Since $f$ is an immersion, so is $H(\cdot,\xi)$ for sufficiently small $\xi$, thus $H(\cdot,\xi)$ stays injective on $U\backslash D_M$ where $U$ is a neighbourhood of the diagonal $D_M \subset M \times M$. On the other hand, since $H(\cdot,\xi)$ interpolates $f$ at $p \in \Lambda$, it stays an injective immersion on $\Omega$. Thus, $H(\cdot,\xi)$ is an injective holomorphic Legendrian immersion $H(\cdot,\xi)\colon M \to \C^3$ for a generic $\xi$ near the origin in $\C^{3N}$, i.\ e.\ an embedding, since $M$ is compact.
\end{proof}

\begin{rmk}
	Note that approximation with interpolation by embeddings is still possible if $M$ is some admissible set in an open Riemann surface $\r$ since the closure $\overline{W}$ of some regular neighbourhood $W$ of such an $M$ is then a compact bordered Riemann surface.
\end{rmk}

\subsection{Controlling the boundary behaviour}

After a minor deformation of the starting data, the next lemma essentially follows from the proof of \cite[Lemma 5.2]{Alarcon2017}. Let $\|\cdot\|_\infty$ denote the $\max$ norm in $\C^{2n+1}$, given by
\[
\|(\bx,\by,z)\|_\infty = \max\{|x_1|,|y_1|,\ldots,|x_n|,|y_n|,|z|\}.
\]

\begin{lm}
	\label{bdenlarge}
	Let $R_1$ and $R_2$ be smoothly bounded compact domains in an open Riemann surface $\r$ such that $R_1 \subset \Int R_2$ and $R_1$ is a deformation retract of $R_2$. Suppose we are given
	\begin{enumerate}[label=(\roman*)]
		\item a generalised Legendrian curve $f \colon  R_1 \to \C^{2n+1}$ of class $\a^r(R_1)$ for some $r \in \N$,
		
		\item a finite set  $\Lambda \subset \Int R_2$ such that $\Lambda = \Lambda' \cup \Lambda''$ where $\Lambda' \subset \Int R_1$ and $\Lambda'' \subset \Int R_2 \backslash R_1$,
		
		\item a holomorphic Legendrian curve $\phi\colon  U \to \C^{2n+1}$, defined on a neighbourhood $U$ of $\Lambda''$,
		
		\item a function $m \colon  \Lambda \to \N \cup \{0\}$, and
		
		\item positive numbers $C,\rho > 0$.		
	\end{enumerate}
	If $f$ satisfies $\|f(p)\|_\infty > \rho$ everywhere on $bR_1$ and $\|\phi(p)\|_\infty > \rho$ holds for every $p \in \Lambda''$, then $f$ may be approximated in the $\c^r(R_1)$-topology by a generalised Legendrian curve $F\colon  R_2 \to \C^{2n+1}$ of class $\a^r(R_2)$ such that
	\begin{enumerate}[label=(\Roman*)]
		\item $\|F(p)\|_\infty > \rho$ everywhere on $R_2 \backslash \Int R_1$,
		\item $\|F(p)\|_\infty > \rho + C$ everywhere on $b R_2$,
		\item $F$ agrees with $f$ to order at least $m(p)$ at every point $p \in \Lambda'$, and
		\item $F$ agrees with $\phi$ to order at least $m(p)$ at every point $p \in \Lambda''$.
	\end{enumerate}
	If $\dd f(p) \neq 0$ for all $p \in \Lambda' \backslash \{m=0\}$ and $\dd\phi(p) \neq 0$ for all $p \in \Lambda'' \backslash \{m=0\}$, then $F$ may be made an immersion. Moreover, if the map $\tilde{f}\colon  \Lambda \to \C^{2n+1}$, given by $\tilde{f}(p) = f(p)$ for $p \in \Lambda$ and $\tilde{f}(p) = \phi(p)$ for $p \in \Lambda''$, is injective, then $F$ may be made an embedding.
\end{lm}

\begin{proof}
	Write $f$ in coordinates as $f = (X_1,\ldots,X_n,Y_1,\ldots,Y_n,Z)$, so that $\pr_i \circ f = X_i$ if $1\leq i \leq n$, $\pr_i \circ f = Y_k$ if $i = n+k$, and $\pr_{2n+1} \circ f = Z$, where $\pr_i\colon  \C^{2n+1}\to\C$ denotes the projection onto the $i$-th coordinate. By assumption, the set $\overline{R_2 \backslash R_1}$ consists of pairwise disjoint annuli. For simplicity, assume there is only one such annulus. The construction may be otherwise repeated on every boundary component of $R_2 \backslash R_1$.
	
	For every $p \in \Lambda''$ choose a smoothly bounded closed disk $O_p \Subset U$ such that $O_p$ does not intersect $R_1$ for every $p \in \Lambda''$ and the family of disks $\{O_p\}_{p \in \Lambda''}$ is pairwise disjoint. Set $O = \cup_p O_p$. Now, connect every $O_p$ with $bR_1$ by a smooth Jordan arc $\gamma_p$ with one endpoint $u_p \in bR_1$ and the other endpoint $o_p \in bO_p$ such that both of these intersections are transverse, the relative interior of $\gamma_p$ is contained in $\Int R_2 \backslash (R_1 \cup O)$, and the family $\{\gamma_p\}_{p \in \Lambda''}$ is pairwise disjoint.

	Note that the set $S = R_1 \cup \bigcup_{p\in \Lambda''}(\gamma_p \cup O_p)$ is admissible in $\Int R_2$. For every $p \in \Lambda''$ there is a number $1\leq i(p) \leq 2n+1$ such that $|\pr_{i(p)} \circ f(u_p)| > \rho$ and a number $1\leq j(p) \leq 2n+1$ such that $|\pr_{j(p)}\circ \phi| > \rho$ on $O_p$, after shrinking $O_p$ and adjusting the arc $\gamma_p$ if necessary such that the above properties still hold. Denote by $\mathbf{z}$ an arbitrary point in $\C^{2n+1}$. The set $W_p = \{\mathbf{z} : |\pr_{i(p)}(\mathbf{z})| > \rho, \, |\pr_{j(p)}(\mathbf{z})|>\rho\}$ is open and connected in $\C^{2n+1}$, thus there exists a smooth path $\widetilde{\Gamma}_p \colon  \im \gamma_p \to W_p$ such that $\widetilde{\Gamma}_p$ agrees with $f$ at $u_p$ and with $\phi$ at $o_p$ to order at least $r$. Now, by \cite[Lemma A.6]{Alarcon2017}, $\widetilde{\Gamma}_p$ may be approximated with a smooth Legendrian path $\Gamma_p$, taking values in $W_p$, agreeing with $f$ at $u_p$ and with $\phi$ at $o_p$ to order $r$. We extend $f$ to a generalised Legendrian curve $\tilde{f}\colon S\to \C^{2n+1}$ by
	\[
	\tilde{f}(u) =
	\begin{cases}
		f(u); & u \in R_1\\
		\Gamma_p(u); & u \in \im\gamma_p \\
		\phi(u); & u \in O.
	\end{cases}
	\]
	By construction, $\tilde{f}$ is of class $\a^r(S)$, thus, using Lemma \ref{finiteInterpolation}, it may be approximated with a holomorphic Legendrian curve $f'\colon S' \to \C^{2n+1}$, defined on some compact neighbourhood $S'$ of $S$ in $\Int R_2$ that deformation retracts onto $S$, such that $f'$ agrees with $f$ to order at least $m(p)$ at every point $p \in \Lambda'$ and $f'$ agrees with $\phi$ to order at least $m(p)$ at every point $p \in \Lambda''$. Moreover, by taking a good enough approximation and shrinking $S'$ if necessary, we may ensure $\|f'\|_\infty > \rho$ holds everywhere on $S'$.
	
	The rest of the proof is almost identical to the proof of \cite[Theorem 5.1]{Alarcon2017}, we thus only sketch the general idea. One first splits the boundary $bS'$ of $S'$ into a collection of arcs $\alpha_j$ where $j \in \Z_L= \Z/l\Z$ for some $l \in \N$ such that for every $j$ there is an $i \in \Z_{2n+1}$ such that $|\pr_i \tilde{f}|> \rho$ on $\alpha_j$. We then decompose $\Z_l$ into (possibly empty) pairwise disjoint subsets $I_i$ for $i \in \Z_{2n+1}$ such that for every $i$ and every $j \in \Z_l$ we have $|\pr_i \circ f|> \rho$ on $\alpha_j$. Label by $p_j$ the unique point in $\alpha_j \cap \alpha_{j+1}$ and connect $p_j$ to $bR_2$ by a smooth Jordan arc $\gamma_j$ with the other enpoint, labeled $q_j$, in $bR_2$. Then extend $\tilde{f}$ over the arcs $\gamma_j$ to a generalised Legendrian curve, still labeled $\tilde{f}$, such that $|\pr_i \circ \tilde{f}| > \rho$ on $\gamma_j$ and $|\pr_i \circ \tilde{f}(q_j)|>\rho +C$ for every $j \in I_i$ and every $i \in \Z_{2n+1}$. Label the arcs in $bR_2$ connecting $q_j$ to $q_{j+1}$ by $\beta_j$ and note that the union $\alpha_j \cup \gamma_j \cup \beta_j \cup \gamma_{j+1}$ is the boundary of a disk $\Omega_j$ in $R_2$ (see Figure \ref{defprcd}).
	
	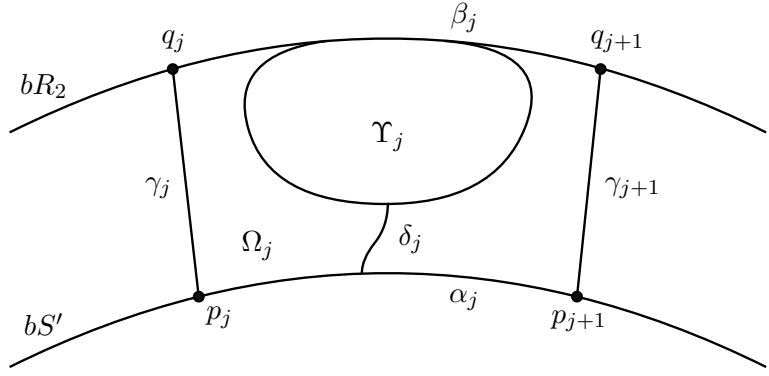
\begin{figure}[h]
		\centering
		\begin{tikzpicture}
			\node[anchor=south west,inner sep=0] (image) at (0,0) {\includegraphics[height=4.5cm]{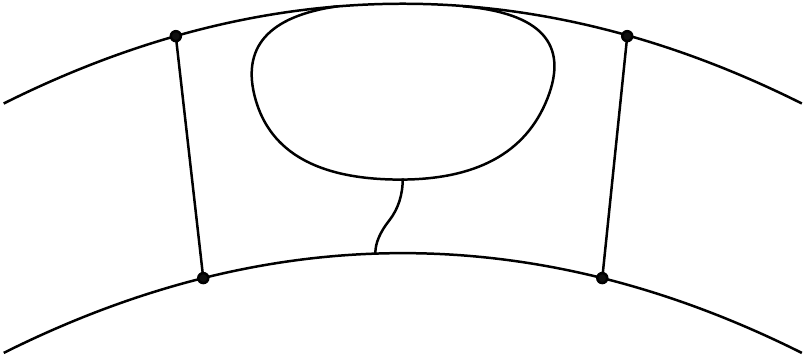}};
			\begin{scope}[x={(image.south east)},y={(image.north west)}]
				\node at (0.28,0.17) {$p_j$};
				\node at (0.75, 0.16) {$p_{j+1}$};
				\node at (0.6, 0.22) {$\alpha_j$};
				\node at (0.6, 1.05) {$\beta_j$};
				\node at (0.5, 0.7) {$\Upsilon_j$};
				\node at (0.22, 0.98) {$q_j$};
				\node at (0.8, 0.98) {$q_{j+1}$};
				\node at (0.2, 0.55) {$\gamma_j$};
				\node at (0.82, 0.55) {$\gamma_{j+1}$};
				\node at (0.33, 0.38) {$\Omega_j$};
				\node at (0.53, 0.40) {$\delta_j$};
				\node at (0.05, 0.15) {$bS'$};
				\node at (0.05, 0.85) {$bR_2$};
			\end{scope}
		\end{tikzpicture}
		\caption{Deformation procedure on $\Omega_j$.}
		\label{defprcd}
	\end{figure}

	We then approximate $\tilde{f}$ by a holomorphic Legendrian curve $g\colon R_2 \to \C^{2n+1}$ Such that
	$|\pr_i \circ g| > \rho$ on $\gamma_j \cup \alpha_j \cup \gamma_{j+1}$ and $|\pr_i\circ g(q_j)|>\rho +C$ for every $j \in I_i$, $i \in \Z_{2n+1}$. Start with $i=1$. For every $j \in I_i$ find a smoothly bounded closed disk $\Upsilon_j$ does not intersect $b\Omega_j$ except possibly in the relative interior of $\beta_j$ such that $|\pr_i\circ g|> \rho$ on $\overline{\Omega_j \backslash \Upsilon_j}$, which is possible by continuity of $g$. Connect $\Upsilon_j$ to $bS'$ by a smooth Jordan arc $\delta_j$ having one endpoint in $bS'$, the other in $b\Upsilon_j$, such that both these intersections are transverse and such that its relative interior is contained in the interior of the set $\overline{\Omega}_j \backslash \Upsilon_j$. Let $S_j = S'  \cup \bigcup_{j \in \Z_l\backslash I_i}\overline{\Omega}_j$ and $S_j' = S_j \cup \bigcup_{j \in I_i}(\delta_j\cup \Upsilon_j)$. Extend $g|_{S_j}$ over $\delta_j \cup \Upsilon_j$ to a generalised Legendrian curve $g_i\colon S_j'\to \C^{2n+1}$ such that $|\pr_{i+1} \circ g_i|>\rho+C$ on $\Upsilon_j$ and $|\pr_{i+1}\circ g|>\rho$ on $\delta_j$ and such that $\pr_i \circ g_i = \pr_i \circ g$ holds on $S_j'$. Note that the latter component is just the restriction of the holomorphic function $\pr_i \circ g$ to $S_j'$. By Lemma \ref{outsideInterpolation} we may approximate $g_i$ by a holomorphic Legendrian curve $G_i\colon R_2 \to \C^{2n+1}$ such that $\pr_i \circ G_i = \pr_i \circ g_i$ while interpolating $g_i$ at the points $p \in \Lambda$ to order at least $m(p)$. If the approximation is sufficiently good, $G_i$ satisfies $\|G_i\|_\infty > \rho$ on $\Omega_j$ and $\|G_i\|_\infty > \rho + C$ on $\beta_j$ for every $j \in I_i$, $|\pr_k\circ G_i|>\rho$ on $\alpha_j, \gamma_j$ and $|\pr_k\circ G_i(q_j)|>\rho+C$ for every $j \in I_k$, $k\neq i$.

	After this deformation procedure is repeated for every $i \in \Z_l$, one obtains a holomorphic Legendrian curve $G=G_{2n+1}\colon R_2 \to \C^{2n+1}$, satisfying the required properties. If needed, one may then approximate $G$ with an immersion or an embedding. If this approximation is close enough, properties (I) and (II) will hold, completing the proof.
\end{proof}

\section{Proof of Theorem \ref{mergelyan}}
	Let $(K_j)_{j=1}^\infty$ be a normal exhaustion of $\mathcal{R}$ by compact $\Or$-convex sets. By compactness of $S$, there exists a $j_0 \in \N$ such that $S \Subset K_{j_0}$. Relabel the sets $K_j$ so that the indexing starts at $K_{j_0}$, that is, obtain a normal exhaustion of the form
	$$
	S := K_0 \Subset K_1 \Subset K_2 \Subset \cdots \Subset K_j \Subset \cdots
	$$
	We may assume without loss of generality that $K_1$ deformation retracts onto $S$ (otherwise take $K_1$ to be the closure of some regular neighbourhood of $S$ and relabel accordingly). By slightly indenting the sets at their boundary without changing their topology (i.\ e.\ without introducing any new holes, so that the indented sets stay Runge) we can ensure $\Lambda \cap bK_j = \emptyset$ for all $j \in \N$, i.~e.\ the discrete set $\Lambda$ does not meet the boundary of any $K_j$. We further require that for all $j$, exactly one of the following conditions holds:
	\begin{enumerate}[label = \roman*)]
		\item $K_{j-1}$ is a deformation retract of $K_{j}$;
		\item $\chi(K_{j}\backslash \Int K_{j-1}) = -1$, where $\chi$ denotes the Euler characteristic, and $K_{j}\backslash \Int K_{j-1}$ contains no point in $\Lambda$.
	\end{enumerate}

	Note that one may always obtain such an exhaustion by fixing a Morse strictly subharmonic $\psi\colon  \r \to \R$ on $\r$ and then define the sets $K_j$ to be the sublevel sets $K_j = \{\psi \leq c_j\}$ of suitably chosen regular values $\{c_j\}_{j\in \N}$ of such a function. Then condition i) corresponds to having no critical values between $c_{j-1}, c_{j}$ and condition ii) corresponds to having exactly one critical value between $c_{j-1}$ and $c_{j}$.
	
	Recall the notation $\Lambda' = \Lambda \cap \Int S$ and $\Lambda'' = \Lambda \cap \r \backslash S$. We construct a sequence $\{f_j\}_{j \in \N}$ of holomorphic Legendrian curves $f_j\colon K_j\to \C^{2n+1}$, satisfying the properties
	\begin{enumerate}[label={\Alph*$_j$:}, ref={\Alph*$_j$}]
		\item
		\label{small}
		 $\|f_j-f_{j-1}\|_{\c^r(K_{j-1})} < \epsilon_j$,
		\item 
		\label{outsidejet}
		$f_j-\phi_p$ has a zero of order at least $m(p)$ at $p$ for all $p \in \Lambda_j'' = \Lambda'' \cap ( \Int K_{j} \backslash K_{j-1})$,
		\item 
		\label{insidejet}
		$f_j-f_{j-1}$ has a zero of order at least $m(p)$ at $p$ for all $p \in \Lambda_j'=\Lambda \cap K_{j-1}$,
		\item 
		\label{immersion}
		$f_j$ is an immersion if $\dd f(p) \neq 0$ for all $p \in \Lambda'$ and $\dd\phi(p) \neq 0$ for all $p \in \Lambda''$,
		\item 
		\label{embedding}
		$f_j$ is an embedding if $\tilde{f}|_\Lambda$, defined in the statement of the theorem, is injective.
	\end{enumerate}
	
	Suppose such a sequence exists. Properties \ref{small} ensure the sequence $f_j$ converges to a holomorphic map $F\colon \mathcal{R} \to \C$ for a suitably chosen sequence $(\epsilon_j)_{j\in \N}$, i.\ e.\ such that the series $\sum_k\epsilon_k$ converges. By continuity of derivatives of $f_j$, the limit $F$ is Legendrian while the properties \ref{outsidejet} and \ref{insidejet} ensure $F-\phi$ has a zero of order at least $m(p)$ at $p$ for all $p \in \Lambda''$. Moreover, the property C$_1$ guarantees that $F-f$ has a zero of order at least $m(p)$ at $p$ for all $p \in \Lambda'$.
	
	Let $\epsilon > 0$ and define
	$$
	\epsilon_j = \frac{1}{2^{j+1}} \min \left\{1, \,  \epsilon \right\}.
	$$
	Now we begin with the induction process. Set $f_0 = f$ and suppose the functions $f_0, \ldots, f_{j-1}$ have been constructed. Note that $\Lambda_j' = \Lambda_{j-1}''$, since $\Lambda \cap bK_j = \emptyset$ for all $j$.
	
	\textbf{Case i):} If $K_{j-1}$ is a deformation retract of $K_{j}$, then $K_{j-1}$ is Runge in some regular neighbourhood $U_{j}$ of $K_{j}$ in $\r$, which also deformation retracts onto $K_{j-1}$. By Lemma \ref{outsideInterpolation} we may assume $f_{j-1}$ is defined on some Runge admissible subset $S_{j-1} \subset U_j$ and agrees with $\phi$ to order at least $m(p)$ for all $p \in \Lambda_j''$. Lemma \ref{finiteInterpolation} then furnishes a generalised Legendrian curve $f_j\colon K_j \to \C^{2n+1}$ satisfying the properties A$_j$, B$_j$ and C$_j$. If the condition in \ref{immersion} holds, Lemma \ref{finiteInterpolation} ensures $f_j$ can be chosen an immersion. If the condition in \ref{embedding} holds, one first uses Lemma \ref{finiteInterpolation} to obtain an immersion $f_j\colon K_j \to \C^{2n+1}$ with the correct jets at the points $p \in \Lambda \cap K_j$ and then uses Theorem \ref{borderedInterpolation} to correct $f_j$ to an embedding (still denoted $f_j$), noting that $K_j$ is a smoothly bounded compact domain in $\r$, thus a compact bordered Riemann surface.
	
	\textbf{Case ii):} If $\chi(K_{j}\backslash \Int K_{j-1}) = -1$, then there exists a smooth embedded arc $\lambda\colon  I \to \Int K_{j}$ with endpoints in $bK_{j-1}$ and otherwise disjoint from $K_{j-1}$ such that $\lambda(I)$ does not intersect $\Lambda$ and $K_{j-1}' = K_{j-1} \cup \Lambda$ is an admissible subset in some regular neighbourhood $U_j$ of $K_j$ in $\r$ and $K_j$ deformation retracts onto $K_{j-1}'$. We can then extend $f_{j-1}\colon K_{j-1} \to \C^{2n+1}$ to a Legendrian curve $f_{j-1}' \colon  K_{j-1}' \to \C^{2n+1}$ of class $\mathcal{A}^r(K_{j-1}')$. Using $K_{j-1}'$ instead of $K_{j-1}$ the situation is the same as in Case i), thus one obtains the curve $f_j$ using the same arguments as above.
	
	It now rests to prove that the sequence $f_j$ satisfying \ref{immersion} for all $j$ converges to an immersion and, moreover, if \ref{embedding} is satisfied for all $j$, that the limit map is injective. Suppose first \ref{immersion} holds for all $j$. If $u \in \r$, then $u \in K_j$ for some $j \geq 0$, but then
	\begin{align*}
		\|\dd F(u)\| &\geq \|\dd f_j(u)\| - \|\dd f_j(u)-\dd F(u)\| \geq  \|\dd f_j(u)\| - \sum_{k=j+1}^\infty \|\dd f_k(u)-\dd f_{k-1}(u)\| \geq \\
		&\geq \min_{u \in K_j}\|\dd f_j(u)\| \left(1 -\sum_{k=j+1}^\infty\frac{1}{2^{k+1}} \right) > 0,
	\end{align*}
	thus $F$ is an immersion, if all $f_j$ are, since in that case the above minimum is strictly greater than zero. This holds by properties \ref{immersion}.
	
	Suppose now, in addition, that \ref{embedding} holds for all $j$. By the Cauchy estimates, there exists a number $\delta_j > 0$ such that any holomorphic map $g\colon  \r \to \C^{2n+1}$ which is $\delta_j$-close to $f_j$ on $K_j$ in the $\c^0 (K_j)$-topology is an embedding on $K_j$. Let
	\[
	\epsilon_{j+1} = \frac{1}{2^{j+1}}\min\{1,\delta_j,\epsilon_j\}.
	\]
	Any two points $q,q' \in \r$ are contained in some $K_{j_0}$ for a big enough $j_0$, but then
	\[
	\|\tilde{f}-f_{j_0}\|_{K_{j_0}} \leq 
	\sum_{j \geq j_0} \|f_{j+1} - f_j\|_{K_{j_0}}
	\leq \sum_{j \geq j_0} \epsilon_j
	\leq \sum_{j\geq j_0} \frac{\delta_{j}}{2^{j+1}}
	\leq \frac{\delta_{j_0}}{2^{j_0+1}} \leq
	\delta_{j_0},
	\]
	thus $F|_{K_{j_0}}$ is an embedding. Since $q,q'$ were arbitrary, $F$ is injective.

	We now explain how to make the approximating curve proper. By assumption the map $\phi\colon  \Lambda'' \to \C^{2n+1}$ is proper. Relabel the sets $K_j$ in the normal exhaustion for $\r$ such that $\Lambda \cap \{\|\phi\|_\infty \leq j\} \subset \Int K_j$ holds for every $j \in \N$. By properties i) and ii) above, it follows that	for every $j \in \N$ the set $K_{j+1} \backslash \Int K_j$ has finitely generated first homology group. As before we construct a sequence of approximating Legendrian curves $f_j\colon  K_j \to \C^{2n+1}$ satisfying A$_j$ -- E$_j$, but now we additionally require
	\begin{enumerate}[label=\Alph*$_j$)]
		\setcounter{enumi}{5}
		\item if $\|f_{j-1}\|_\infty > \rho$ on $bK_{j-1}$ for some $\rho > 0$, then $\|f_j\|_\infty > \rho + 1$ on $bK_j$.
	\end{enumerate}

	We approximate $f\colon  S \to \C^{2n+1}$ with a holomorphic Legendrian embedding $f'\colon  U \to \C^{2n+1}$, defined on a relatively compact neighbourhood $U$ of $S$ in $\r$ which deformation retracts onto $S$ and contains no point in $\Lambda''$, such that $f'$ agrees with $f$ to order at least $m(p)$ at every point $p \in \Lambda'$. By shrinking $U$ if necessary, we may assume $f'$ is defined on $\overline{U}$ and has no zeroes on $bU$, hence there exists a number $\rho>0$ such that $\|f'\|_\infty > \rho$ on $bU$. Put $f_0 = f\colon  \overline{U} \to \C^{2n+1}$.
	
	Now suppose $f_0,f_1,\ldots,f_j$ have been constructed and assume $\|f_j\|_\infty > \rho$ on $bK_j$. By a previous observation the set $K_{j+1}$ deformation retracts onto $K_j'$, where the latter is obtained from $K_j$ by attaching to it finitely many smooth Jordan arcs with endpoints in $bK_j$. Furthermore, we may assume none of these arcs intersects $\Lambda''$. Thus, $K_j$ is an admissible set in some open neighbourhood $U_j$ of $K_j'$ such that $U_j \backslash K_j$ does not contain any point in $\Lambda$. We approximate $f_j$ with a holomorphic Legendrian curve $f_j'\colon U_j \to \C^{2n+1}$ that still satisfies properties A$_j$ -- F$_j$. Choose a slightly smaller set $U_j' \Subset U_j$ which deformation retracts onto $K_j'$. Now use the Lemma \ref{bdenlarge} to approximate $f_j'$ with a generalised Legendrian curve $f_{j+1} \colon  K_{j+1} \to \C^{2n+1}$ satisfying the properties A$_{j+1}$ -- E$_{j+1}$.
	
	A similar computation as above shows that the sequence of curves $f_j$ converges uniformly on compact sets in $\r$ to a holomorphic Legendrian curve $F\colon  \r \to \C^{2n+1}$ satisfying the properties in the theorem. Namely, as above, one shows the curve $F$ is an injective immersion, hence an embedding, by properties F$_j$.
	
\section{Proof of Theorem \ref{carleman}}
	The theorem follows from previous results using a simple induction procedure. By Lemma \ref{bddexhaustion} there exists a normal exhaustion $\{K_j\}_{j=1}^\infty$ of $\r$ by (compact) $\o(\r)$-convex sets $K_j$ such that for all $j$ the set $S_j := S \cup K_j$ is $\o(\r)$-convex. By slightly deforming the sets $K_j$ if necessary we may furthermore achieve that for all $j$:
	\begin{enumerate}[label = \alph*)]
		\item $\Lambda_j := \Lambda \cap K_j \subset \Int K_j$ and
		\item $bK_j \cap S \subset E\backslash K$ and the intersection $bK_j \cap E$ is transverse.
	\end{enumerate}
	
	We first show how to construct the approximating curve in the general case, that is without the immersivity, injectivity or properness assumptions on $f$. We inductively construct generalised Legendrian curves $f_j\colon  S_j \to \C^{2n+1}$ satisfying the conditions:
	\begin{enumerate}[label= \Alph*$_j$:, ref=\Alph*$_j$]
		\item 
		\label{Aj}
		$|f_j(q)-f_{j-1}(q)| < \epsilon_j$ for all $q \in S_{j-1}$ where
		$$
		\epsilon_j = \frac{1}{2^j} \min \left\{\inf_{q \in S\cap K_{j}} \varepsilon(q), 1\right\},
		$$
		\item 
		\label{Bj}
		$f_j - f_{j-1}$ has a zero of order at least $k(q)$ for all $q \in \Lambda'_j := \Lambda \cap K_{j-1}$,
		\item 
		\label{Cj}
		$f_j - \phi$ has a zero of order at least $k(p)$ for all $p \in \Lambda_{j}'' := \Lambda \cap (K_j \backslash \Int K_{j-1})$ (note that $\Lambda_j = \Lambda_j' \cup \Lambda_j''$)
		\item 
		\label{Dj}
		$f_j = f$ on $S \backslash K_{j+1}$
	\end{enumerate}

	Suppose we have constructed such a sequence. Properties \ref{Aj} ensure the sequence $(f_j)_{j=0}^\infty$ converges to a holomorphic Legendrian curve $F\colon  \r \to \C^{2n+1}$ which is $\varepsilon$-close to $f$ in the $\c^0(S)$-topology. Properties \ref{Bj} and \ref{Cj} ensure the approximating curve $F$ satisfies the properties (I) and (II) from the theorem whereas properties \ref{Dj} are needed for the induction step which we explain next.

	Set $S_0 =S$, $f_0 = f$, and note that the conditions A$_0$ to D$_0$ are void. Now suppose $f_1, \ldots, f_{j-1}$ have been constructed. Note that $S_{j-1}'=S_{j-1} \cap K_j$ is by construction a compact Runge admissible set in some smoothly bounded relatively compact open neighbourhood $U_j$ of $K_j$ in $\r$ and $\Lambda''_j \subset K_j = \Lambda_j'' \cap \Int K_j$ is a finite set not intersecting $S_{j-1}'$. We suppose in addition that $K_j$ is a deformation retract of $U_j$. By attaching a finite number (possibly none) of pairwise disjoint smooth embedded Jordan arcs to $S_{j-1}'$ with endpoints in $bK_{j-1}$ we obtain a compact Runge admissible set $K_{j-1}'$ such that $K_{j-1}'$ is a deformation retract of $K_j$. By a slight deformation of these arcs if necessary, we ensure $K_{j-1}'$ does not intersect $\Lambda_j''$. We then extend $f_{j-1}$ to a generalised Legendrian curve $f_{j-1}'\colon  K_{j-1}' \to \C^{2n+1}$.	By Lemma \ref{outsideInterpolation} we may assume $f_{j-1}'$ is defined on some compact Runge admissible set $K_{j-1}'' \subset K_{j}$ containing $\Lambda_j''$ in its interior. By Lemma \ref{finiteInterpolation} there exists a generalised Legendrian curve $\tilde{f}_j \colon  K_j \to \C^{2n+1}$ satisfying the conditions \ref{Bj} and \ref{Cj} such that $|\tilde{f}_j(q)-f_{j-1}(q)| < \epsilon_j$ for all $q \in S_j \cap K_{j-1}$. 
	
	Note that $bK_j \cap S = bK_j \cap E$ is a finite set $bK_j \cap E = \{p_{j,1},\ldots, p_{j,n_j}\}$ and let $E_{j,i} \subset E$ be the Jordan curve in the collection $E$ containing $p_{j,i}$. For each $p_{j,i}$ let $\epsilon_{j,i}\colon [0,\delta] \to S$ be a parametrisation of $E_{j,i}$ near $p_{j,i}$, such that $\epsilon_{j,i}([0,\delta]) \cap E_{j,i} = \epsilon_{j,i}(0) = p_{j,i}$, i.\ e.\ $\epsilon_{j,i}$ parametrises a subarc of $E_{j,i}$ lying outside of $K_j$ and emanating from $p_{j,i}$. If the approximation of $f_{j-1}$ by $\tilde{f}_j$ is good enough, the line $t \mapsto (1-t/\delta)\tilde{f}(p_{j,i}) + (t/\delta) f(p_{j,i}) \in \C^{2n+1}$ lies in $\Int K_{j+1}$ and is $\c^0$-close to the curve $t\mapsto f(\epsilon_{j,i}(t))$, hence it may be approximated (after smoothing it out near the endpoints) in $\c^0$-topology by a smooth Legendrian arc $\gamma_{j,i}\colon [0,\delta] \to \C^{2n+1}$ such that $\gamma_{j,i}$ agrees with $\tilde{f}\circ \epsilon_{j,i}$ at $t=0$ to order at least 1 and with $f\circ \epsilon_{j,i}$ at $t=\delta$ also to order at least 1 (see \cite[Theorem A.6]{Alarcon2017}). By repeating this procedure for all $p_{j,i}$, $i=1,\ldots,n_j$, we obtain a generalised Legendrian curve $f_j\colon K_j\cup S \to \C^{2n+1}$ of class $\mathcal{A}^1(K_j \cup S)$ such that $f_j$ is $\c^0$ close to $f$ on $S\cap K_{j+1}$ and $f_j = f$ on $S \backslash K_{j+1}$, i.\ e.\ $f_j$ satisfies the properties \ref{Aj}, \ref{Bj}, \ref{Cj} and \ref{Dj}, closing the induction.	
	
	Now suppose the map $\tilde{f} \colon  S \cup O \to \C^{2n+1}$, defined in the statement of the theorem, is an immersion on some neighbourhood of the set $\Lambda$. By Lemma \ref{outsideInterpolation} we may ensure all the approximating maps $f_j$ are also immersions in addition to satisfying properties \ref{Aj} - \ref{Dj}. Furthermore, by the Cauchy estimates and compactness of $K_j$ there exists for each $j$ a number $\nu_j > 0$ such that any map $g\colon K_j \to \C^{2n+1}$ of class $\mathcal{A}^1(K_j)$ is an immersion whenever $|f_j(q)-g(q)| < \nu_j$. In the property \ref{Aj} we additionally require that
	\[
	\epsilon_j = \frac{1}{2^j} \min \{1, \nu_{j-1}, \inf_{u \in K_j}\varepsilon(u)\}
	\]
	As above, the limit map $F = \lim_{j\to \infty}f_j$ is a holomorphic Legendrian curve $F\colon  \r \to \C^{2n+1}$ which is $\varepsilon$-close to $f$. Any point $u \in \r$ is contained in some $K_j$ for $j$ big enough, thus
	\begin{align*}
		\|\dd F(u)\| &\geq \|\dd f_j(u)\| - \|\dd f_j(u)-\dd F(u)\| \geq  \|\dd f_j(u)\| - \sum_{k=j+1}^\infty \|\dd f_k(u)-\dd f_{k-1}(u)\| \geq \\
		&\geq \nu_j
		\left(1 -\sum_{k=j+1}^\infty\frac{1}{2^{k+1}} \right) > 0,
	\end{align*}
	since $f_j$ is an immersion and $K_j$ is compact. Thus, $F$ is an immersion.
	
	Suppose now the map $\tilde{f}|_\Lambda$ is also injective. By Theorem \ref{borderedInterpolation} we can insure each map $f_j|_{K_j}\colon  K_j \to \C^{2n+1}$ is a holomorphic Legendrian embedding. It follows again from the Cauchy estimates that for each $j$ there is a $\delta_j > 0$ such that any holomorphic immersion $g\colon K_j \to \C^{2n+1}$ of class $\mathcal{A}^r(K_j)$ which is $\delta_j$-close to $f_j$ on $K_j$ is en embedding. To construct an injective limit map it thus suffices to put
	\[
	\epsilon_j = \frac{1}{2^j} \min \{1,\delta_{j-1}, \inf_{u \in K_j}\varepsilon(u)\}
	\]
	in the induction procedure. The limit map $F = \lim_j f_j$ is then an embedding when restricted to an arbitrary compact set $K \subset \r$. Thus, $F$ is injective.
	
	Finally, suppose $\tilde{f}_{S \cup \Lambda''}$ is proper. Relabel the sets $K_j$ of the previously constructed normal exhaustion such that in addition the following hold:
	\begin{enumerate}[label=\alph*)]
		\setcounter{enumi}{2}
		\item $\{q \in S \cup \Lambda'': \;\|\tilde{f}(q)\|_\infty \leq j\} \subset K_j$, and
		\item $\|\tilde{f}(q)\|_\infty > j$ for every $q \in bK_j \cap S$.
	\end{enumerate}
	When constructing the curves $f_j\colon  K_j \to \C^{2n+1}$ we require that these curves satisfy the properties A$_j$ -- D$_j$ in addition to:
	\begin{enumerate}[label=\Alph*$_j$:]
		\setcounter{enumi}{4}
		\item $\|f_j(q)\|_\infty > j$ for every $q \in bK_j$.
	\end{enumerate}
	We now explain how to ensure the above property for every $j \in \N$. The extension $f_{j-1}'$ of $f_{j-1}$ over the arcs attached to $K_{j-1}$ may be constructed such that $\|f_{j-1}'(q)\| > j-1$ holds for every point $q$ lying in the newly added arcs by using \cite[Theorem A.6]{Alarcon2017} and the assumption E$_{j-1}$. Then, using Lemma \ref{bdenlarge}, we approximate $f_{j-1}'$ with a generalised Legendrian curve $\tilde{f}_j \colon  K_j \to \C^{2n+1}$ such that $\|\tilde{f}_j(q)\|_\infty > j$ holds for every $q \in bK_j$. This implies $\|\tilde{f}_j(p_{j,i})\|_\infty > j$ for every $i=1,\ldots,n_j$, where $p_{j,i}$ are points in $bK_j \cap S$ as above. By c), we may construct the arcs $\gamma_{j,i}$ as above such that $\|\gamma_{j,i}(t)\|_\infty > j$ for every $t \in [0,\delta]$, obtaining a generalised Legendrian curve $f_j\colon  S_j \to \C^{2n+1}$ satisfying the property F$_j$ thus closing the induction. 
	
	\bigskip
	\noindent
	\textsc{Acknowledgements:} The author is supported by grant MR-54828 from ARRS, Republic of Slovenia, associated to the research program P1-0291 \emph{Analysis and Geometry}. The author would like to thank F. Forstnerič for his advice, support and guidance.

	\bigskip


\bibliographystyle{abbrv}                 

\bibliography{document}

\end{document}